\documentclass[11pt,a4paper]{amsart}
\pdfoutput=1
\usepackage[left=25mm, top=1.5in, bottom=1.5in]{geometry}
\usepackage[pagewise]{lineno}
\usepackage{amsaddr}
\usepackage{amsthm}
\usepackage{amsmath}
\usepackage{amssymb}
\usepackage{mathrsfs}
\usepackage{amsfonts}
\usepackage{mathtools}
\usepackage{tikz}
\usetikzlibrary{matrix,decorations.pathmorphing, positioning}
\usetikzlibrary{arrows,automata,backgrounds,petri}
\usetikzlibrary{shapes.multipart}
\usetikzlibrary{decorations.pathreplacing}
\usetikzlibrary{calc, trees}
\usepackage{enumitem}
\usepackage{verbatim}
\usepackage{array}
\usepackage{multirow}
\usepackage{hhline}
\usepackage{makecell}
\usepackage{caption}
\usepackage{float}

\usepackage[initials]{amsrefs}
\usepackage{hyperref} 

\restylefloat{table}

\theoremstyle{definition}
\newtheorem{definition}{Definition}[section]
\newtheorem{thm}{Theorem}[section]

\newtheorem{prop}{Proposition}[thm]
\newtheorem{lem}[thm]{Lemma}
\newtheorem{exam}{Example}[section]
\newtheorem*{rmk}{Remark}

\newcommand{\LLL}{\mathscr{L}}
\newcommand{\x}{\mathbf{x}}
\newcommand{\y}{\mathbf{y}}
\newcommand{\s}{\mathbf{s}}
\newcommand{\St}{\mathbf{St}}
\newcommand{\A}{\mathcal{A}}
\newcommand{\C}{\mathcal{C}}
\newcommand{\lf}{\varliminf\limits}
\newcommand{\I}{\textrm{(i)}}
\newcommand{\II}{\textrm{(ii)}}
\newcommand{\III}{\textrm{(iii)}}
\newcommand{\FP}[1][1]{u'_{#1}}
\newcommand{\FPW}[1][1]{v'_{#1}}
\newcommand{\FPG}[2][1]{u_{#1,#2}}
\newcommand{\FPGW}[2][1]{v_{#1,#2}}

\allowdisplaybreaks

\DeclareMathOperator{\rep}{rep}

\begin{document}

\title{The spectrum of The exponents of repetition}
\author{\footnotesize Deokwon Sim}

\maketitle

\begin{abstract}
    For an infinite word $\mathbf{x}$, Bugeaud and Kim introduced a new complexity function $\text{rep}(\mathbf{x})$ which is called the exponent of repetition of $\mathbf{x}$. 
    They showed $1\le \text{rep}(\mathbf{x}) \le \sqrt{10}-\frac{3}{2}$ for any Sturmian word $\mathbf{x}$. 
    Ohnaka and Watanabe found a gap in the set of the exponents of repetition of Sturmian words. 
    For an irrational number $\theta\in(0,1)$, let
    \[ \mathscr{L}(\theta):=\{\text{rep}(\mathbf{x}):\textrm{$\mathbf{x}$ is an Sturmian word of slope $\theta$}\}.
    \]
    In this article, we look into $\mathscr{L}(\theta)$. 
    The minimum of $\mathscr{L}(\theta)$ is determined where $\theta$ has bounded partial quotients in its continued fraction expression. 
    In particular, we find out the maximum and the minimum of $\mathscr{L}(\varphi)$ where $\varphi:=\frac{\sqrt{5}-1}{2}$ is the fraction part of the golden ratio. 
    Furthermore, we show that the three largest values are isolated points in $\mathscr{L}(\varphi)$ and the fourth largest point is a limit point of $\mathscr{L}(\varphi)$.
\end{abstract}

\section{Introduction}
A word $\x$ over $\A$ is a finite or infinite sequence of elements of a finite set $\A$. 
An element of $\A$ is called a \emph{letter}.
For each integer $n\ge1$, $p(n,\x)$ is defined by the number of distinct subwords of length $n$ appearing in the word $\x$ and is called by the \emph{subword complexity} of $\x$.
Morse and Hedlund showed that an infinite word is eventually periodic if and only if its subword complexity is bounded \cite{MH40}.
Thus, a non-eventually periodic word $\x$ with the smallest subword complexity satisfies $p(n,\x)=n+1$ for all $n\ge 1$.
\begin{definition}
A \emph{Sturmian word} is an infinite word $\x$ over $\A=\{0,1\}$ satisfying $p(n,\x)=n+1$ for all $n\ge 1$. 
\end{definition}

Let us recall a classical result on Sturmian words \cite{L02}*{Chapter 2}.
For a Sturmian word $\x=x_{1}x_{2}\dots$, there exist an irrational number $\theta\in(0,1)$ and a real number $\rho$ such that either \[x_{n}=\lfloor{(n+1)\theta+\rho}\rfloor-\lfloor{n\theta+\rho}\rfloor\textrm{ for all }n\ge1\] or
\[x_{n}=\lceil{(n+1)\theta+\rho}\rceil-\lceil{n\theta+\rho}\rceil\textrm{ for all }n\ge1.\] 
Conversely, for an irrational number $\theta\in(0,1)$ and a real number $\rho$, the infinite words defined by
\[\s_{\theta, \rho}:=s_{1}s_{2}\dots,\text{ }\s'_{\theta, \rho}:=s'_{1}s'_{2}\dots
\]
are Sturmian where $s_{n}=\lfloor{(n+1)\theta+\rho}\rfloor-\lfloor{n\theta+\rho}\rfloor$ and $s'_{n}=\lceil{(n+1)\theta+\rho}\rceil-\lceil{n\theta+\rho}\rceil$ for $n\ge1$.
They are called the Sturmian words \emph{of slope $\theta$ and intercept $\rho$}.

Sturmian words have been studied in many different areas \cites{AB98, BKLN, K96, L02}.
Various complexities have been looked into characterize Sturmian words such as Cassaigne's recurrence function, rectangle complexity \cites{C97, BV00}. 
Bugeaud and Kim recently suggested a new complexity function $r(n,\x)$ and characterized Sturmian words and eventually periodic words in terms of $r(n,\x)$ \cite{BK19}.

\begin{definition}
Given an infinite word $\x=x_{1}x_{2}\dots$,
let $r(n,\x)$ denote the length of the smallest prefix in which some subword of length $n$ occurs twice. More precisely,
\[r(n,\x):=\min\{m:x_{j}x_{j+1}\dots x_{j+n-1}=x_{m-n+1}x_{m-n+2}\dots x_{m}\textrm{ for some $1\le j \le m-n$}\}.
\]
The \emph{exponent of repetition} of $\x$ is defined by 
\begin{align*}
\rep({\x}) := \liminf_{n\to\infty}{\frac{r(n,\x)}{n}}.
\end{align*}
\end{definition}
By definition, $r(n+1,\x)\ge r(n,\x)+1$ for $n\ge1$ and $\rep(c\x)=\rep(\x)$ for any finite word $c$.
Theorem 2.3 and 2.4 in \cite{BK19} says that
for an infinite word $\x$, the following statements hold:
\begin{enumerate}[label=(\roman*)]
    \item $\x$ is eventually periodic if and only if $r(n,\x)\le 2n$ for all sufficiently large integers $n$.
    \item $\x$ is a Sturmian word if and only if $r(n,\x)\le 2n+1$ for all $n\ge 1$ and equality holds for infinitely many $n$.
\end{enumerate}

A number of spectrums such as Lagrange spectrum and Markoff spectrum have been studied for a long time. 
Spectrums connected with Diophantine approximation are one of main topics in Number theory, Combinatorics, Dynamics \cites{CF89, F12, AK21}. 
In particular, the set of the irrationality exponents appears in many references \cites{B12, AR09}. 
Theorem 4.5 in \cite{BK19} says that for a Sturmian word $\x=x_1 x_2 \dots$, an integer $b\ge2$, and a Sturmian number $r_{\x}:=\sum_{k\ge1}\frac{x_k}{b^k}$, the exponent of repetition $\rep(\x)$ and the irrationality exponent of $r_{\x}$, denoted by $\sup\{\mu:|r_{\x}-\frac{p}{q}|<\frac{1}{q^{\mu}}\text{ for infinitely many $p,q$}\}$, have a special relation:
$\mu (r_{\x})=\frac{\rep(\x)}{\rep(\x)-1}$ where $\mu (r_{\x})$ is the irrationality exponent of $r_{\x}$.
Thus, we look into the spectrum of the exponents of repetition of Sturmian words.
From now on, we call $r_{\x}$ a \emph{Sturmian number of slope $\theta$} if $\x$ is a Sturmian word of slope $\theta$.

Let us set up some notations on continued fraction expansion. Any irrational number $\theta$ has the unique continued fraction expansion of the form 
\[[a_{0};a_{1},a_{2},\dots]:=a_{0}+\dfrac{1}{a_{1}+\dfrac{1}{a_{2}+\cdots}}
\]
where $a_{0}\in\mathbb{Z}$ and $a_{n}\in\mathbb{Z}_{>0}$  for $n\ge1$.
Each $a_{n}$ is called a \emph{partial quotient} of $\theta$.
We say that $\theta$ has \emph{bounded partial quotients} if all of partial quotients of $\theta$ are bounded.
Let $k\ge0$.
We call $[a_{0};a_{1}, a_{2}, \dots, a_{k}]$ the \emph{$k$th principal convergent} of $\theta$.
Denote $\frac{p_{k}}{q_{k}}=[a_{0};a_{1}, a_{2}, \dots, a_{k}]$.
Let $p_{-1}=1$, $q_{-1}=0$.
It is known that $p_{k+1}=a_{k+1}p_{k}+p_{k-1}$ and $q_{k+1}=a_{k+1}q_{k}+q_{k-1}$ for all $k\ge0$. 

Let us denote by $\St$ the set of Sturmian words over $\A=\{0,1\}$ and by $\rep(\St)$ the set of $\rep(\x)$ for all $\x \in \St$. 
In \cite{BK19}, Bugeaud and Kim found the maximum of $\rep(\St)$ and a necessary and sufficient condition for the maximum of $\rep(\St)$.
They also gave an alternative proof of \cite{AB11}*{Proposition 11.1}, which is a sufficient condition for the minimum of $\rep(\St)$.

\begin{thm}(\cite{BK19}, Theorem 3.3 and 3.4)
For any Sturmian word $\x$,
$$
1 \le \rep(\x) \le \sqrt{10}-\frac{3}{2}
$$
Moreover, a Sturmian word $\x$ satisfies $\rep(\x)=r_{\textrm{max}}:=\sqrt{10}-\frac{3}{2}$ if and only if its slope is of the form $[0;a_{1}, a_{2}, \dots, a_{K}, \overline{2,1,1}]$ for some $K$. 
A Sturmian word $\x$ satisfies $\rep(\x)=1$ if its slope has unbounded partial quotients.
\end{thm}

On the other hand, Ohnaka and Watanabe recently discovered the second largest value $r_{1}$ of $\rep(\St)$ and proved that $r_{1}$ is a limit point of $\rep(\St)$.
\begin{thm}(\cite{OW}, Theorem 0.4)
For $r_{1}=\frac{48+\sqrt{10}}{31}=1.650\dots$,
$$(r_{1},r_{\textrm{max}})\cap \rep(\St)=\emptyset$$
Moreover, $r_{1}$ is a limit point of $\rep(\St)$.
\end{thm}

For any irrational number $\theta\in(0,1)$, let $\LLL(\theta)$ be the set of the exponents of repetition of Sturmian words of slope $\theta$, i.e.
\[
\LLL(\theta) = \{ \mathrm{rep}({\x}) \, : \, \x \text{ is a Sturmian word of slope $\theta$}\}.
\]
In this article, we mainly study $\LLL(\theta)$.
Theorem 3.3 in \cite{BK19} gives $\rep(\x)=1$ where the slope of $\x$ has unbounded partial quotients. 
We find the minimum of $\LLL(\theta)$ where $\theta$ has bounded partial quotients.

\begin{thm}\label{TheMinOfSpec}
Let $\theta=[0;a_{1},a_{2},\dots]$ have bounded partial quotients.
We have
$$
\min\LLL(\theta) = \lf_{k\to\infty} \, [1; 1+ a_k, a_{k-1}, a_{k-2}, \dots, a_1].
$$
\end{thm}

In particular, for $\varphi:=[0;\overline{1}]$, we find the minimum $\mu_{\min}$ and the maximum $\mu_{\max}$ of $\LLL(\varphi)$. 
In addition, we determine Sturmian words $\x$ satisfying $\rep(\x)=\mu_{\min}$ or $\mu_{\max}$ in terms of the locating chain of $\x$, to be defined in Section \ref{Def_and_Exam} and \ref{GoldenRatio}.

\begin{thm}\label{TheMaxMinForGoldenRatio}
Let $\x$ be a Sturmian word of slope $\varphi$. 
Then, $\mu_{min} \le \rep(\x) \le\mu_{max}.$
Moreover, the locating chain of $\x$ is $u \overline{a}$ or $v\overline{b}$ for some finite words $u,v$ if and only if $\rep(\x) = \mu_{max}.$
We have $\x\in S_d$ for any $d\ge1$ if and only if $\rep(\x) = \mu_{min}.$
\end{thm}

We call an interval $(c,d)$ a \emph{maximal gap} in $\LLL(\varphi)$ if $(c,d)\cap \LLL(\varphi)=\emptyset$ and there exist Sturmian words $\x$, $\y$ such that $\rep(\x)=c$, $\rep(\y)=d.$ 
We show that the three largest points $\mu_{max}, \mu_{2}, \mu_{3} $ are isolated points in $\LLL(\varphi)$ and the fourth largest point $\mu_{4}$ is an accumulation point of $\LLL(\varphi)$. 
Moreover, we determine Sturmian words $\x$ satisfying $\rep(\x)=\mu_{2}, \mu_{3},\text{ or }\mu_{4}$ in terms of the locating chain of $\x$

\begin{thm}\label{TheRightGap}
The intervals
$
\left( \mu_{2} , \mu_{max} \right),
\left(\mu_{3}, \mu_{2}\right)$
are maximal gaps in $\LLL(\varphi)$. Moreover, the locating chain of $\x$ is $u\overline{ab}$ for some finite word $u$ if and only if $\rep(\x) = \mu_{2}$. 
The locating chain of $\x$ is $v\overline{b^{2}a^{2}}$ for some finite word $v$ if and only if $\rep(\x)=\mu_{3}$.
\end{thm}

\begin{thm}\label{TheThirdRightGap}
The interval $\left(\mu_{4}, \mu_{3}\right)$ is a maximal gap in $\LLL(\varphi)$.
Moreover, $\rep(\x)=\mu_{4}$ if and only if $\x \in S_{2}\cap S_{3}^{c}$ satisfies the following two conditions:\\
1) The locating chain of $\x$ is $u(b^2 a^2)^{e_1}ba(b^2 a^2)^{e_2}ba\dots$ for some finite word $u$ and integers $e_i \ge 1$.\\
2) $\limsup\limits_{i\ge1}\{e_{i}\}=\infty$.\\
Furthermore, $\mu_{4}$ is a limit point of $\LLL(\varphi)$.
\end{thm}

Indeed, 
\begin{align*}
    \mu_{max}&=1+\varphi=1.618\dots, \hspace{2mm}
    \mu_{2}=1+{2\varphi^3}=1.472\dots, \hspace{2mm}
    \mu_{3}=1+\frac{\varphi^2 (\varphi^4+\varphi^2+1)}{\varphi^5+\varphi^3+1}=1.440\dots,\\
    \mu_{4}&=1+\frac{1-\varphi^6}{1+2\varphi-2\varphi^7+\varphi^{11}}=1.435\dots, \hspace{2mm}
    \mu_{min}=1+\varphi^2=1.382\dots.
\end{align*}

\begin{figure}[H]
    \centering
    \begin{tikzpicture}[scale=2.5]

        \draw[<->] (0,0) -- (6,0);
        \filldraw [black] (0.5,0) circle (0.5pt) node[below] {$\mu_{min}$};
        \filldraw [black] (1.45,0) circle (0.3pt);
        \filldraw [black] (1.5,0) circle (0.3pt);
        \filldraw [black] (1.55,0) circle (0.3pt);
        \filldraw [black] (1.6,0) circle (0.5pt) node[below] {$\mu_{4}$};
        \filldraw [black] (1.8,0) circle (0.5pt) node[below] {$\mu_{3}$};
        \filldraw [black] (2.3,0) circle (0.5pt) node[below] {$\mu_{2}$};
        \filldraw [black] (5.3,0) circle (0.5pt) node[below] {$\mu_{max}$};
        
        \draw [decorate,decoration={brace,amplitude=3pt},xshift=0pt,yshift=0.5pt]
        (1.6,0) -- (1.8,0) node[black,midway,yshift=7pt] 
        {gap};
        \draw [decorate,decoration={brace,amplitude=4pt},xshift=0pt,yshift=0.5pt]
        (1.8,0) -- (2.3,0) node[black,midway,yshift=8pt] 
        {gap};
        \draw [decorate,decoration={brace,amplitude=5pt},xshift=0pt,yshift=0.5pt]
        (2.3,0) -- (5.3,0) node[black,midway,yshift=9pt] 
        {gap};
    
    \end{tikzpicture}
    \caption{$\mu_{max},\mu_{2},\mu_{3},\mu_{4},\mu_{min}$ in $\LLL(\varphi)$}
    \label{fig:my_label}
\end{figure}
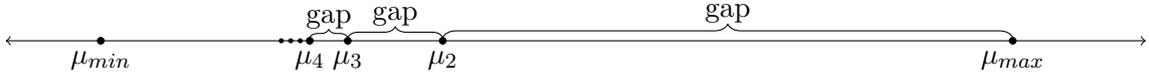

For $\alpha\in\{\mu_{max}, \mu_{2}, \mu_{3}, \mu_{4}, \mu_{min}\}$, we obtain the cardinality of the set of Sturmian words $\x$ satisfying $\rep(\x)=\alpha$.

\begin{prop}\label{Cardinality_Of_Sturmian words}
For $\alpha\in\{\mu_{max},\mu_{2},\mu_{3}\}$, there are only countably many Sturmian words $\x$ of slope $\varphi$ satisfying $\rep(\x) = \alpha.$
For $\beta\in\{\mu_{4}, \mu_{min}\}$, there are uncountably many Sturmian words $\x$ of slope $\varphi$ satisfying $\rep(\x) = \beta.$
\end{prop}

The article is organized as follows. 
In Section \ref{Def_and_Exam}, we define basic notations in Sturmian words and the locating chain of a Sturmian word. 
In Section \ref{GeneralSlope}, we investigate the set $\{n:r(n,\x)=2n+1\}$ for a Sturmian word $\x$.
We prove Theorem \ref{TheMinOfSpec} in Theorem \ref{Proof_TheMinOfSpec}.
In Section \ref{GoldenRatio}, we study $\rep(\x)$ for a Sturmian word $\x$ of slope $\varphi$.
We establish Theorem \ref{TheMaxMinForGoldenRatio}, \ref{TheRightGap}, and \ref{TheThirdRightGap} in Theorem \ref{Proof_TheMaxMinForGoldenRatio}, \ref{Proof_TheRightGap}, and \ref{Proof_TheThirdRightGap}, respectively. 
We prove Proposition \ref{Cardinality_Of_Sturmian words} in Proposition \ref{proof_Cardinality_Of_Sturmian_words}.

At the final stage of writing the paper, the author found an independent result by Bugeaud and Laurent \cite{BL}:
Theorem \ref{TheMinOfSpec} is equivalent to Theorem 2.6 in \cite{BL}. 
For a given irrational number $\theta\in(0,1)$, Bugeaud and Laurent found the maximum of the irrationality exponents of Sturmian numbers of slope $\theta$. 
Using Theorem 4.5 in \cite{BK19} and Section 4 in \cite{B27}, Theorem \ref{TheMinOfSpec} is induced.

\section{Basic definitions and examples}\label{Def_and_Exam}

Let us begin by some notations in Sturmian words of slope $\theta$ where $\theta\in(0,1)$ is an irrational number.
Let $\theta:=[0 ;a_1 , a_2 , \dots]$. We define a sequence $\{M_{k}\}_{k\ge 0}$ in the following way: 
Let us define $M_0:=0$, $M_1:=0^{a_{1}-1}1$, and $M_{k+1}:=M_{k}^{a_{k+1}}M_{k-1}$ for $k\ge 1 $.
Let us denote the numerator and denominator of the $k$th principal convergent of $\theta$ by $p_{k}, q_{k}$, respectively.
We denote the length of $M_k$ by $|M_k|$.
Since $|M_{1}|=a_{1}=q_{1}$, $|M_{0}|=1=q_{0}$, and $|M_{k+1}|=a_{k+1}|M_{k}|+|M_{k-1}|$ for $k\ge1$, it is obvious that $|M_k|=q_{k}$. 
The \emph{characteristic Sturmian word} of slope $\theta$ is obtained by
$\mathbf{c}_{\theta}:=\lim\limits_{k\to\infty}{M_k}.$

For a non-empty finite word $V$, let us denote by $V^{-}$ the word $V$ with the last letter removed.
Let $k\ge1.$
Note that $M_{k}M_{k-1}$ and $M_{k-1}M_{k}$ are identical, except for the last two letters \cite{L02}*{Proposition 2.2.2}. 
Let $\widetilde{M}_{k}:=M_{k}M_{k-1}^{--}=M_{k-1}M_{k}^{--}$.
The last two letters of $M_{k}M_{k-1}$ is 01 (resp., 10) if and only if the last two letters of $M_{k-1}M_{k}$ is 10 (resp., 01). 
We denote by $D_{k}$, $D'_{k}$ the last two letters such that $M_{k}M_{k-1}=\widetilde{M}_{k}D_{k}$, $M_{k-1}M_{k}=\widetilde{M}_{k}D'_{k}$, respectively.

In this section, let $\x = x_1 x_2 \dots $ be a Sturmian word of slope $\theta$. By Lemma 7.2 in \cite{BK19}, for any $k\ge 1$, there exists a unique word $W_k$ satisfying one of the following cases
\begin{enumerate}[label=(\roman*)]
    \item $\x=W_{k}M_{k}\widetilde{M_k}\dots$, where $W_k$ is a non-empty suffix of $M_k$,
    \item $\x=W_{k}M_{k-1}M_{k}\widetilde{M_k}\dots$, where $W_k$ is a non-empty suffix of $M_k$,
    \item $\x=W_{k}M_{k}\widetilde{M_k}\dots$, where $W_k$ is a non-empty suffix of $M_{k-1}$.
\end{enumerate}
For case (i) and case (ii), there exist $q_{k}$ non-empty suffices of $M_k$. For case (iii), there exist $q_{k-1}$ non-empty suffices of $M_{k-1}$. 
Lemma 7.2 in \cite{BK19} also gives that all the $(2q_k+q_{k-1})$ cases are mutually exclusive.
For each $k\ge1$, we say that $\x$ \emph{belongs to case (i), (ii), (iii) at level $k$} if $W_{k}$ satisfies case (i), (ii), (iii), respectively.
We denote by $\C_{k}^{\I}$, $\C_{k}^{\II}$, $\C_{k}^{\III}$ the set of Sturmian words which belong to case (i), (ii), (iii) at level $k$, respectively.
For each $\x$, we have an infinite sequence of (i), (ii) and (iii) for which case $\x$ belongs to at level $1, 2, \dots$, called the \emph{locating chain} of $\x$. 
In the locating chain of $\x$, let $u^d:=\underbrace{uu\dots u}_{d}$ where $u$ is a finite word of (i),(ii),(iii). 

\begin{exam}
Let $\x = \mathbf c_{\varphi}$, i.e. the characteristic Sturmian word of slope $\varphi=[0;\overline{1}]$.
Since $\x$ starts with $M_{k+1}M_{k}M_{k+1}=M_{k}M_{k-1}M_{k}\widetilde{M}_{k}D_{k}$ for any $k\ge1$, $\x\in\C_{k}^{\II}$ and $W_{k}=M_{k}$ for all $k\ge1$. 
Hence, the locating chain of $\x$ is $\overline{\II}$.
\end{exam}

\begin{exam}
Let $\x = 1 \mathbf c_{\varphi}$. 
Since $\x$ starts with $1M_{k+2}=1M_{k}M_{k-1}M_{k}=1M_{k}\widetilde{M}_{k}D'_{k}$ for any $k\ge1$, $W_{k}=1$ for $k\ge1$. 
Moreover, $\x\in\C_{k}^{\I}$ if $k$ is odd, and $\x\in\C_{k}^{\III}$ if $k$ is even. 
Hence, the locating chain of $\x$ is $\overline{\I\III}$.
\end{exam}

\begin{exam}
Let $\x=10101M_{4}M_{5}\dots$.
Since $\x$ starts with $W_{1}M_{0}M_{1}=101$, $\x\in\C_{1}^{\II}$ and $W_{2}=W_{1}M_{0}=10$. 
Since $\x$ starts with $W_{2}M_{2}\widetilde{M}_{2}=10101$, $\x\in\C_{2}^{\I}$ and $W_{3}=W_{2}=10$.
Since $\x$ starts with $W_{3}M_{3}M_{4}=W_{3}M_{3}\widetilde{M}_{3}D_{3}$, $\x\in\C_{3}^{\III}$ and $W_{4}=W_{3}=10$.
Since $\x$ starts with $W_{4}M_{3}M_{4}\widetilde{M}_{4}$, $\x\in\C_{4}^{\II}$ and $W_{5}=W_{4}M_{3}=10101$.
Moreover, for $k\ge5$, $\x$ starts with $W_{k}M_{k-1}M_{k}\widetilde{M}_{k}$ and $W_{k+1}=W_{k}M_{k-1}$ where $W_{k}$ is a non-empty suffix of $M_{k}$. 
Hence, $\x\in\C_{k}^{\II}$ for $k\ge5$.
Therefore, the locating chain of $\x=\II\I\III\overline{\II}$.
\end{exam}

Theorem 2.4 in \cite{BK19} says that $r(n,\x)\le 2n+1$ for all $n\ge1$ and equality holds for infinitely many $n$.
Let
$$
\Lambda(\x):= \{ n \in \mathbb N : r(n,\x) = 2n+1 \}.
$$
We have $\Lambda(\x)=\{ n_1, n_2, \dots \} $ for an increasing sequence $\{n_{i}\}_{i\ge1}$.
From Lemma 5.3 in \cite{BK19}, $r(n,\x)\le 2n$ implies $r(n,\x)\le r(n-1,\x)+1$. 
Thus,
\[
r(n, \x) = r(n-1, \x) +1 \ \text{ if } n \notin \Lambda(\x).
\]
Hence, the sequence $\left\{\dfrac{r(n,\x)}{n}\right\}_{n\ge1}$ is decreasing on each interval $[n_{i},n_{i+1}-1]$. It gives
\[
\mathrm{rep}(\x) = \liminf_{i \to \infty} \left( 1 + \frac{n_i}{n_{i+1}} \right).
\]

\section{The exponents of repetition of Sturmian words}\label{GeneralSlope}

Let $\x$ be a Sturmian word of slope $\theta:=[0;a_1 ,a_2 , \dots].$
For $k\ge1$, there is a relation between cases which $\x$ belongs to at level $k$ and $k+1$.

\begin{lem}\label{LocatingChainRule}
Let $k\ge1$. The following statements are satisfied.

(1) If $\x\in \C_{k}^{\I}$ and $\x\in \C_{k+1}^{\I}\cup\C_{k+1}^{\II}$, then $W_{k+1}=W_{k}M_{k}^t M_{k-1}$ for some $1\le t\le a_{k+1}-1$.

(2) If $\x\in \C_{k}^{\I}$ and $\x\in \C_{k+1}^{\III}$, then $W_{k+1}=W_{k}$.

(3) If $\x\in \C_{k}^{\II}$, then $\x\in \C_{k+1}^{\I}\cup\C_{k+1}^{\II}$ and $W_{k+1}=W_{k}M_{k-1}$.

(4) If $\x\in \C_{k}^{\III}$, then $\x\in \C_{k+1}^{\I}\cup\C_{k+1}^{\II}$ and $W_{k+1}=W_{k}$.

\end{lem}

\begin{proof}
In this proof, for all $k\ge1$, let $W_{k}$ be the unique prefix of $\x$ defined in which case $\x$ belongs to at level $k$.
Note that $\widetilde{M}_{k}$ is a prefix of $M_{k+1}^{--}$ by definition.

(1) Let $\x\in \C_{k+1}^{\I}\cup\C_{k+1}^{\II}$.
Note that $\x$ starts with $W_{k+1}M_{k}\widetilde{M}_{k}$ for the suffix $W_{k+1}$ of $M_{k+1}$. 
If $W_{k+1}$ is a non-empty suffix of $M_{k-1}$, then $\x\in\C_{k}^{\III}$. 
It is a contradiction. 
If $W_{k+1}=W'_{k}M_{k-1}$ for some non-empty suffix $W'_{k}$ of $M_{k}$, then $\x$ starts with $W'_{k}M_{k-1}M_{k}\widetilde{M}_{k}$. 
Thus, $\x\in\C_{k}^{\II}$. 
It is a contradiction. 
Hence, $a_{k+1}>1$ and $W_{k+1}=W''_{k}M_{k}^{t}M_{k-1}$ for some $1\le t\le a_{k+1}-1$ and some non-empty suffix $W''_{k}$ of $M_{k}$. 
Consequently, $\x$ starts with $W''_{k}M_{k}\widetilde{M}_{k}$. 
By the uniqueness of $W_{k}$, $W_{k}=W''_{k}$. It gives $W_{k+1}=W_{k}M_{k}^{t}M_{k-1}$.

(2) Let $\x\in \C_{k+1}^{\III}$. We have $\x=W_{k+1}M_{k+1}\widetilde{M}_{k+1}$ for a non-empty suffix $W_{k+1}$ of $M_{k}$.
Since $M_{k+1}\widetilde{M}_{k+1}=M_{k}M_{k+1}^{--}D_{k+1}M_{k+1}^{--}$, $\x$ starts with $W_{k+1}M_{k}\widetilde{M}_{k}$.
By the uniqueness of $W_{k}$, $W_{k+1}=W_{k}$.

(3) Let $\x\in \C_{k}^{\II}$.
Note that $\x$ starts with $W_{k}M_{k-1}M_{k}\widetilde{M}_{k}$ for a non-empty suffix $W_{k}$ of $M_{k}$.
Assume that $\x\in \C_{k+1}^{\III}$.
Since $\x$ starts with $W_{k+1}M_{k+1}\widetilde{M}_{k+1}$ for the suffix $W_{k+1}$ of $M_{k}$, $\x$ starts with $W_{k+1}M_{k}\widetilde{M}_{k}$. 
Hence, $\x\in \C_{k}^{\I}$. 
It is a contradiction. 
Hence, $\x\in \C_{k+1}^{\I}\cup\C_{k+1}^{\II}$.
Thus, $\x$ starts with $W_{k+1}M_{k}\widetilde{M}_{k}$ for the suffix $W_{k+1}$ of $M_{k+1}$.
If $W_{k+1}$ is a non-empty suffix of $M_{k-1}$, then $\x\in \C_{k}^{\III}$.
It is a contradiction.
If $W_{k+1}=W'_{k}M_{k}^{t}M_{k-1}$ for some $1\le t\le a_{k+1}-1$ and some non-empty suffix $W'_{k}$ of $M_{k}$, then $\x$ starts with $W'_{k}M_{k}\widetilde{M}_{k}$.
Thus, $\x\in \C_{k}^{\I}$.
It is a contradiction.
Hence, $W_{k+1}=W'_{k}M_{k-1}$ for some non-empty suffix $W'_{k}$ of $M_{k}$. 
By the uniqueness of $W_{k}$, $W_{k+1}=W_{k}M_{k-1}$.

(4) Let $\x\in \C_{k}^{\III}$.
Note that $\x$ starts with $W_{k}M_{k}\widetilde{M}_{k}$ for the suffix $W_{k}$ of $M_{k-1}$.
Assume that $\x\in \C_{k+1}^{\III}$.
Since $\x$ starts with $W_{k+1}M_{k+1}\widetilde{M}_{k+1}$ for the suffix $W_{k+1}$ of $M_{k}$, $\x$ starts with $W_{k+1}M_{k}\widetilde{M}_{k}$. 
Hence, $\x\in \C_{k}^{\I}$. 
It is a contradiction. 
Hence, $\x\in \C_{k+1}^{\I}\cup\C_{k+1}^{\II}$.
Thus, $\x$ starts with $W_{k+1}M_{k}\widetilde{M_{k}}$ for the suffix $W_{k+1}$ of $M_{k+1}$.
If $W_{k+1}=W'_{k}M_{k}^{t}M_{k-1}$ for $0\le t\le a_{k+1}-1$ and some non-empty suffix $W'_{k}$ of $M_{k}$, then $\x$ starts with $W'_{k}M_{k-1}M_{k}\widetilde{M_{k}}$ or $W'_{k}M_{k}\widetilde{M_{k}}$.
Thus, $\x\in \C_{k}^{\I}\cup\C_{k}^{\II}$.
It is a contradiction.
Hence, $W_{k+1}=W'_{k-1}$ for some non-empty suffix $W'_{k-1}$ of $M_{k-1}$. 
By the uniqueness of $W_{k}$, $W_{k+1}=W_{k}$.
\end{proof}

\begin{table}[H]\label{TableLocatingChainRule}
    \begin{tabular}{ c | c | c }
      \hline
      $k$   &  $k+1$  & The relation between $W_{k+1}$ and $W_{k}$\\
      \hline \hline
      \multirow{2}{*} {case (i)} & case (i) case (ii)  &   $W_{k+1} = W_{k}{M_{k}}^{t}M_{k-1}$ $(1\le t\le a_{k+1}-1)$ \\
      \hhline{~--}& case (iii)  &   $W_{k+1} = W_k$ \\
      \hline
      case (ii) & case (i) case (ii) &  $W_{k+1} = W_k M_{k-1}$ \\
      \hline
      case (iii) & case (i) case (ii) & $W_{k+1} = W_k$ \\
      \hline
    \end{tabular}
    \captionof{table}{The relation between $W_{k+1}$ and $W_{k}$}
\end{table}


Let 
\begin{align*}
    &u_{t,k}=tq_{k}+q_{k-1}-1, &&v_{t,k}=|W_{k}|+tq_{k}+q_{k-1}-1,\\
    &u'_{k}=q_{k+1}-1,  &&v'_{k}=|W_{k}|+q_{k+1}-1.\end{align*}
The following lemma shows that all of elements in $\Lambda(\x)\cap[\FPG{1},\infty)$ are expressed in terms of $q_{k}$'s and $|W_{k}|$'s.

\begin{lem}\label{PositionToComputeGeneralcase}
Let $k\ge1$.

(1) If $\x\in\C_{k}^{\I}\cap\C_{k+1}^{\I}$, then 
\begin{align*}
    &\Lambda(\x) \cap [\FPG{k}, \FPG{k+1}-1]
    =
    \begin{cases}
    \{ \FPGW[t]{k}, \FP[k]\}&\textrm{for } t=a_{k+1}-1,\\
    \{ 
    \FPGW[t]{k}, \FPG[t+1]{k},
    \FPGW[t+1]{k}, \FP[k] 
    \}&\textrm{for } t<a_{k+1}-1
    \end{cases}
\end{align*}
where $t$ satisfies $W_{k+1}=W_{k}{M_{k}}^{t}M_{k-1}$.

(2) If $\x\in\C_{k}^{\I}\cap\C_{k+1}^{\II}$, then
\begin{align*}
    &\Lambda(\x) \cap [\FPG{k}, \FPG{k+1}-1]
    = \{ \FPGW[t]{k}, \FPG[t+1]{k}, \FPGW[t+1]{k}\}
\end{align*}
where $t$ satisfies $W_{k+1}=W_{k}{M_{k}}^{t}M_{k-1}$.

(3)
If $\x\in\C_{k}^{\I}\cap\C_{k+1}^{\III}\cap\C_{k+2}^{\I}$, then
\[ 
\Lambda(\x) \cap [\FPG{k}, \FPG{k+2}-1] = \{ \FPW[k], \FP[k+1] \}.
\]

(4) If $\x\in\C_{k}^{\I}\cap\C_{k+1}^{\III}\cap\C_{k+2}^{\II}$, then
\[ 
\Lambda(\x) \cap [\FPG{k}, \FPG{k+2}-1] = \{ \FPW[k] \}.
\]

(5) If $\x\in\C_{k}^{\II}\cap\C_{k+1}^{\I}$, then
\begin{align*}
    &\Lambda(\x) \cap [\FPG{k}, \FPG{k+1}-1]
    =
    \begin{cases}
    \{ \FPG{k}\}&\textrm{for } a_{k+1}=1,\\
    \{ 
    \FPG{k}, \FPGW{k}, \FP[k]
    \}&\textrm{for } a_{k+1}>1.
    \end{cases}
\end{align*}

(6) If $\x\in\C_{k}^{\II}\cap\C_{k+1}^{\II}$, then
\[ 
\Lambda(\x) \cap [\FPG{k}, \FPG{k+1}-1] = \{  \FPG{k}, \FPGW{k} \}.
\]
\end{lem}

\begin{proof}
For $\x=x_{1}x_{2}\dots$, let $x_{i}^{j}:=x_{i}x_{i+1}\dots x_{j}.$

(1) Suppose that $\x\in\C_{k}^{\I}\cap\C_{k+1}^{\I}$.
Since $M_{k}M_{k-1}$ is primitive, Lemma 7.1 in \cite{BK19} implies that for
\[\x=W_{k}{M_{k}}\widetilde{M}_{k}\dots=W_{k}M_{k}M_{k-1}{M_{k}}^{--}=W_{k}M_{k-1}{M_{k}}^{--}D_{k}{M_{k-1}}^{--}\dots,\]
the first $q_{k}$ subwords of length $\FPG{k}$ are mutually distinct. 
From $x_{1}^{\FPG{k}}=x_{q_{k}+1}^{\FPG[2]{k}}$,
$r(\FPG{k},\x)=\FPG[2]{k}.$
Note that
\begin{align*}
    \x&=W_{k+1}M_{k+1}\widetilde{M}_{k+1}\dots
    =W_{k}{M_{k}}^{t}M_{k-1}{M_{k}}^{a_{k+1}}M_{k-1}\widetilde{M}_{k+1}\dots\\
    &=W_{k}{M_{k}}^{t+1}{M_{k-1}}^{--}D'_{k}{M_{k}}^{a_{k+1}-1}M_{k-1}\widetilde{M}_{k+1}\dots.
\end{align*}
Since $x_{1}^{\FPGW[t]{k}-1}=x_{q_{k}+1}^{\FPGW[t+1]{k}-1}$, 
$r(\FPGW[t]{k}-1,\x)\le\FPGW[t]{k}-1.$
The fact that $r(n+1,\x)\ge r(n,\x)+1$ for any $n\ge1$ gives 
\[r(n,\x)=n+q_{k}\] 
for $\FPG{k}\le n \le \FPGW[t]{k}-1.$
Moreover, we have
$r(\FPGW[t]{k},\x)\ge r(\FPGW[t]{k}-1,\x)+2.$
Hence,
$r(\FPGW[t]{k},\x)=2\FPGW[t]{k}+1$
by Theorem 2.4 and Lemma 5.3 in \cite{BK19}.
Since $x_{1}^{\FPG[t+1]{k}-1}=x_{\FPGW[t]{k}+2}^{\FPGW[2t+1]{k}+q_{k-1}-1}$, we have
$r(\FPG[t+1]{k}-1,\x)\le \FPGW[2t+1]{k}+q_{k-1}-1.$
The fact that $r(n+1,\x)\ge r(n,\x)+1$ for any $n\ge1$ gives
\[r(n,\x)=n+\FPGW[t]{k}+1\] 
for $\FPGW[t]{k}\le n \le \FPG[t+1]{k}-1.$
Note that
\begin{align*}
    \x&=W_{k+1}M_{k+1}\widetilde{M}_{k+1}\dots
    =W_{k}{M_{k}}^{t}M_{k-1}{M_{k}}^{a_{k+1}}M_{k-1}\widetilde{M}_{k+1}\dots\\
    &=W_{k}{M_{k}}^{t+1}{M_{k-1}}^{--}D'_{k}{M_{k}}^{a_{k+1}-1}M_{k-1}\widetilde{M}_{k+1}\dots
    =W_{k}{M_{k}}^{t}M_{k-1}{M_{k}}^{t+1}M_{k-1}\dots.
\end{align*}
It gives
$r(\FPG[t+1]{k},\x)\ge r(\FPG[t+1]{k}-1,\x)+2.$
Thus, we have 
$r(\FPG[t+1]{k},\x)=2\FPG[t+1]{k}+1$
by Theorem 2.4 and Lemma 5.3 in \cite{BK19}.
If $t=a_{k+1}-1$, then $(t+1)q_{k}+q_{k-1}=q_{k+1}.$
By the argument used at level $k$,
$r(\FPG{k+1}-1,\x)=\FPG[2]{k+1}-1.$
The fact that $r(n+1,\x)\ge r(n,\x)+1$ for any $n\ge1$ gives
\[r(n,\x)=n+\FP[k]+1\] for $\FP[k]\le n \le \FPG{k+1}-1.$ It follows that 
\[
\Lambda(\x) \cap [\FPG{k}, \FPG{k+1}-1] = \{ \FPGW[t]{k}, \FP[k]\}.
\]
Now, let $t<a_{k+1}-1.$
Note that
\begin{align*}
    \x&=W_{k+1}M_{k+1}\widetilde{M}_{k+1}\dots
    =W_{k}{M_{k}}^{t}M_{k-1}{M_{k}}^{a_{k+1}}{M_{k-1}}M_{k}{M_{k+1}}^{--}\dots\\
    &=W_{k}{M_{k}}^{t}M_{k-1}M_{k}{M_{k}}^{a_{k+1}-2}M_{k}{M_{k-1}}M_{k}{M_{k+1}}^{--}\dots\\
    &=W_{k}{M_{k}}^{t+1}{M_{k-1}}^{--}D_{k}{M_{k}}^{a_{k+1}-1}M_{k}{M_{k-1}}^{--}D'_{k}{M_{k+1}}^{--}\dots.
\end{align*}
Since $x_{1}^{\FPGW[t+1]{k}-1}=x_{\FPG[t+1]{k}+2}^{|W_{k}|+tq_{k}+q_{k-1}+((t+2)q_{k}+q_{k-1}-2)}$,
we have 
$r(\FPGW[t+1]{k}-1,\x)\le \FPGW[2t+2]{k}+q_{k-1}-1.$
The fact that $r(n+1,\x)\ge r(n,\x)+1$ for any $n\ge1$ gives
\[r(n,\x)=n+\FPG[t+1]{k}+1\] for $\FPG[t+1]{k}\le n \le \FPGW[t+1]{k}-1.$
Note that
\begin{align*}
    \x&=W_{k+1}M_{k+1}\widetilde{M}_{k+1}\dots
    =W_{k}{M_{k}}^{t}M_{k-1}{M_{k}}^{a_{k+1}}{M_{k-1}}M_{k}{M_{k+1}}^{--}\dots\\
    &=W_{k}{M_{k}}^{t+1}{M_{k-1}}^{--}D'_{k}{M_{k}}^{a_{k+1}-1}M_{k}{M_{k-1}}^{--}D'_{k}{M_{k+1}}^{--}\dots\\
    &=W_{k}{M_{k}}^{t}M_{k-1}{M_{k}}^{t+2}{M_{k-1}}\dots.
\end{align*}
It gives
$r(\FPGW[t+1]{k},\x)\ge r(\FPGW[t+1]{k}-1,\x)+2.$
Hence, we have 
$r(\FPGW[t+1]{k},\x)=2\FPGW[t+1]{k}+1$ from Theorem 2.4 and Lemma 5.3 in \cite{BK19}.
On the other hand, from
\begin{align*}
    \x&=W_{k+1}M_{k+1}\widetilde{M}_{k+1}\dots
    =W_{k}{M_{k}}^{t}M_{k-1}M_{k}{M_{k}}^{a_{k+1}-1}{M_{k-1}}M_{k}{M_{k+1}}^{--}\dots\\
    &=W_{k}{M_{k}}^{t}M_{k-1}M_{k}{M_{k}}^{a_{k+1}}{M_{k-1}}^{--}D'_{k}{M_{k+1}}^{--}\dots,
\end{align*}
we have $x_{|W_{k}|+tq_{k}+q_{k-1}+1}^{|W_{k}|+tq_{k}+q_{k-1}+(a_{k+1}q_{k}+q_{k-1}-2)}=x_{|W_{k}|+(t+1)q_{k}+q_{k-1}+1}^{|W_{k}|+(t+1)q_{k}+q_{k-1}+(a_{k+1}q_{k}+q_{k-1}-2)}$. It gives
$r(\FP[k]-1,\x)\le (\FPGW[t+1]{k}+1)+(\FP[k]-1).$
The fact that $r(n+1,\x)\ge r(n,\x)+1$ for any $n\ge1$ gives
\[r(n,\x)=n+\FPGW[t+1]{k}+1\] for $\FPGW[t+1]{k}\le n \le \FP[k]-1.$
Moreover,
$r(\FP[k],\x)\ge r(\FP[k]-1,\x)+2.$
From Theorem 2.4 and Lemma 5.3 in \cite{BK19},
$r(\FP[k],\x)=2\FP[k]+1.$
Note that $\x\in\C_{k+1}^{\I}$.
By the argument used at level $k$,
$r(\FPG{k+1}-1,\x)=\FPG[2]{k+1}-1.$
The fact that $r(n+1,\x)\ge r(n,\x)+1$ for any $n\ge1$ gives
\[r(n,\x)=n+\FP[k]+1\] for $\FP[k]\le n \le \FPG{k+1}-1.$ 
It follows that 
\[ 
\Lambda(\x) \cap [\FPG{k}, \FPG{k+1}-1] = 
    \left\{
    \FPGW[t]{k}, \FPG[t+1]{k},
    \FPGW[t+1]{k}, \FP[k]
    \right\}.
\]

(2) Suppose that $\x\in\C_{k}^{\I}\cap\C_{k+1}^{\II}$.
Note that
\begin{align*}
    \x=W_{k+1}M_{k}\dots
    =W_{k}{M_{k}}^{t}M_{k-1}M_{k}\dots
    =W_{k}{M_{k}}^{t+1}{M_{k-1}}^{--}D'_{k}\dots.
\end{align*}
Use the argument used at level $k$ in (1). Since $x_{1}^{\FPGW[t]{k}-1}=x_{q_{k}+1}^{\FPGW[t+1]{k}-1}$,
\[r(n,\x)=n+q_{k}\] for $\FPG{k}\le n \le \FPGW[t]{k}-1.$
Note that
\begin{align*}
    \x&=W_{k+1}M_{k}M_{k+1}\widetilde{M}_{k+1}\dots
    =W_{k}{M_{k}}^{t}M_{k-1}M_{k}{M_{k}}^{a_{k+1}}M_{k-1}\widetilde{M}_{k+1}\dots\\
    &=W_{k}{M_{k}}^{t+1}{M_{k-1}}^{--}D'_{k}{M_{k}}^{a_{k+1}}M_{k-1}\widetilde{M}_{k+1}\dots
    =W_{k}{M_{k}}^{t}M_{k-1}{M_{k}}^{t+1}M_{k-1}\dots.
\end{align*}
Since $x_{|W_{k}|+1}^{|W_{k}|+\FPG[t+1]{k}-1}=x_{\FPGW[t]{k}+2}^{\FPGW[t]{k}+1+\FPG[t+1]{k}-1}$,
$r(\FPG[t+1]{k}-1,\x)\le \FPGW[2t+1]{k}+q_{k-1}-1.$
Moreover,
$r(\FPGW[t]{k},\x)\ge r(\FPGW[t]{k}-1,\x)+2.$
From Theorem 2.4 and Lemma 5.3 in \cite{BK19},
$r(\FPGW[t]{k},\x)=2\FPGW[t]{k}+1.$
The fact that $r(n+1,\x)\ge r(n,\x)+1$ for any $n\ge1$ gives
\[r(n,\x)=n+\FPGW[t]{k}+1\] 
for $\FPGW[t]{k}\le n \le \FPG[t+1]{k}-1.$
Moreover,
$r(\FPG[t+1]{k},\x)\ge r(\FPG[t+1]{k}-1,\x)+2.$
From Theorem 2.4 and Lemma 5.3 in \cite{BK19},
$r(\FPG[t+1]{k},\x)=2\FPG[t+1]{k}+1.$
On the other hand, note that
\begin{align*}
    \x&=W_{k+1}M_{k}M_{k+1}\widetilde{M}_{k+1}\dots
    =W_{k}{M_{k}}^{t}M_{k-1}M_{k}{M_{k}}^{a_{k+1}}M_{k-1}M_{k}\dots\\
    &=W_{k}{M_{k}}^{t}{M_{k-1}}M_{k}{M_{k}}^{a_{k+1}+1}{M_{k-1}}^{--}\dots
    =W_{k}{M_{k}}^{t+1}{M_{k-1}}^{--}D'_{k}{M_{k}}^{a_{k+1}+1}{M_{k-1}}^{--}\dots.
\end{align*}
Since $x_{1}^{\FPGW[t+1]{k}-1}=x_{\FPG[t+1]{k}+2}^{\FPG[t+1]{k}+1+\FPGW[t+1]{k}-1}$,
$r(\FPGW[t+1]{k}-1,\x)\le \FPGW[2t+2]{k}+q_{k-1}-1.$
The fact that $r(n+1,\x)\ge r(n,\x)+1$ for any $n\ge1$ gives
\[r(n,\x)=n+\FPG[t+1]{k}+1\] 
for $\FPG[t+1]{k}\le n \le \FPGW[t+1]{k}-1.$
Moreover,
$r(\FPGW[t+1]{k},\x)\ge r(\FPGW[t+1]{k}-1,\x)+2.$
From Theorem 2.4 and Lemma 5.3 in \cite{BK19},
$r(\FPGW[t+1]{k},\x)=2\FPGW[t+1]{k}+1.$
On the other hand, note that
\begin{align*}
    \x&=W_{k+1}M_{k}M_{k+1}\widetilde{M}_{k+1}\dots
    =W_{k}{M_{k}}^{t}M_{k-1}M_{k}{M_{k}}^{a_{k+1}}M_{k-1}M_{k}\dots\\
    &=W_{k}{M_{k}}^{t}M_{k-1}M_{k}{M_{k}}^{a_{k+1}+1}{M_{k-1}}^{--}D'_{k}\dots.
\end{align*}
Since $x_{\FPGW[t]{k}+2}^{\FPGW[t]{k}+1+\FPG[a_{k+1}+1]{k}-1}=x_{\FPGW[t+1]{k}+2}^{\FPGW[t+1]{k}+1+\FPG[a_{k+1}+1]{k}-1}$,
$r(\FPG{k+1}-1,\x)\le \FPG{k+1}-1+\FPGW[t+1]{k}+1.$
The fact that $r(n+1,\x)\ge r(n,\x)+1$ for any $n\ge1$ gives
\[r(n,\x)=n+\FPGW[t+1]{k}+1\] 
for $\FPGW[t+1]{k}\le n \le \FPG{k+1}-1.$
Hence, 
\[\Lambda(\x) \cap [\FPG{k}, \FPG{k+1}-1]=\{\FPGW[t]{k}, \FPG[t+1]{k}, \FPGW[t+1]{k}\}.
\]

(3) Suppose that $\x\in\C_{k}^{\I}\cap\C_{k+1}^{\III}\cap\C_{k+2}^{\I}$.
Since $\x\in\C_{k}^{\I}$, the argument used at level $k$ in (1) yields
$r(\FPG{k},\x)=\FPG[2]{k}.$
Note that
\begin{align*}
    \x
    =W_{k+1}M_{k+1}\widetilde{M}_{k+1}\dots 
    = W_{k}{M_{k}}^{a_{k+1}}M_{k-1}M_{k}{M_{k+1}}^{--}\dots
    = W_{k}{M_{k}}^{a_{k+1}}\widetilde{M}_{k}D'_{k}{M_{k+1}}^{--}\dots.
\end{align*}
Since $x_{1}^{\FPGW[a_{k+1}]{k}-1}=x_{q_{k}+1}^{\FPGW[a_{k+1}]{k}+q_{k}-1}$,
$r(\FPW[k]-1,\x)\le \FPW[k]+q_{k}-1.$
The fact that $r(n+1,\x)\ge r(n,\x)+1$ for any $n\ge1$ gives
\[r(n,\x)=n+q_{k}\] 
for $\FPG{k}\le n \le \FPW[k]-1.$
Moreover,
$r(\FPW[k],\x)\ge r(\FPW[k]-1,\x)+2.$
From Theorem 2.4 and Lemma 5.3 in \cite{BK19},
$r(\FPW[k],\x)=2\FPW[k]+1.$
Note that
\begin{align*}
    \x
    &=W_{k+2}M_{k+2}\widetilde{M}_{k+2}\dots
    = W_{k}M_{k+1}{M_{k+1}}^{a_{k+2}-1}M_{k}\widetilde{M}_{k+2}\dots \\
    &= W_{k}M_{k+1}{M_{k+1}}^{a_{k+2}}{M_{k}}^{--}D'_{k+1}{M_{k+2}}^{--}\dots.
\end{align*}
Since $x_{|W_{k}|+1}^{|W_{k}|+\FPG[a_{k+2}]{k+1}-1}=x_{|W_{k}|+q_{k+1}+1}^{|W_{k}|+\FPG[a_{k+2}+1]{k+1}-1}$,
$r(\FP[k+1]-1,\x)\le|W_{k}|+q_{k+1}+\FP[k+1]-1.$
The fact that $r(n+1,\x)\ge r(n,\x)+1$ for any $n\ge1$ gives
\[r(n,\x)=n+\FPW[k]+1\] 
for $\FPW[k]\le n \le \FP[k+1]-1.$
Moreover,
$r(\FP[k+1],\x)\ge r(\FP[k+1]-1,\x)+2.$
From Theorem 2.4 and Lemma 5.3 in \cite{BK19}, 
$r(\FP[k+1],\x)=2\FP[k+1]+1.$
Since $\x\in\C_{k+2}^{\I}$, the argument used at level $k$ in (1) implies
$r(\FPG{k+2},\x)=\FPG[2]{k+2}.$
The fact that $r(n+1,\x)\ge r(n,\x)+1$ for any $n\ge1$ gives
\[r(n,\x)=n+\FP[k+1]+1\] 
for $\FP[k+1]\le n \le \FPG{k+2}.$
Hence,
\[\Lambda(\x) \cap [\FPG{k}, \FPG{k+2}-1]=\{\FPW[k], \FP[k+1]\}.
\]

(4) Suppose that $\x\in\C_{k}^{\I}\cap\C_{k+1}^{\III}\cap\C_{k+2}^{\II}$.
Use the same argument with (3). 
Since $\x\in\C_{k}^{\I}\cap\C_{k+1}^{\III}$, we have
\[r(n,\x)=n+q_{k}\]
for $\FPG{k}\le n \le \FPW[k]-1$
and
$r(\FPW[k],\x)=2\FPW[k]+1.$
On the other hand, note that
\begin{align*}
    \x
    &=W_{k+2}M_{k+1}M_{k+2}\widetilde{M}_{k+2}\dots
    = W_{k}M_{k+1}{M_{k+1}}^{a_{k+2}}M_{k}\widetilde{M}_{k+2}\dots\\
    &= W_{k}M_{k+1}{M_{k+1}}^{a_{k+2}+1}{M_{k}}^{--}D'_{k+1}{M_{k+2}}^{--}\dots.
\end{align*}
Since $x_{|W_{k}|+1}^{|W_{k}|+\FPG[a_{k+2}+1]{k+1}-1}=x_{|W_{k}|+q_{k+1}+1}^{|W_{k}|+q_{k+1}+\FPG[a_{k+2}+1]{k+1}-1}$,
$r(\FPG{k+2}-1,\x)\le|W_{k}|+q_{k+1}+\FPG{k+2}-1.$
The fact that $r(n+1,\x)\ge r(n,\x)+1$ for any $n\ge1$ gives 
\[r(n,\x)=n+\FPW[k]+1\] 
for $\FPW[k]\le n \le \FPG{k+2}-1.$
Hence,
\[\Lambda(\x) \cap [\FPG{k}, \FPG{k+2}-1]=\{\FPW[k]\}.
\]

(5) Suppose that $\x\in\C_{k}^{\II}\cap\C_{k+1}^{\I}$.
Note that
\begin{align*}
    \x
    &=W_{k}M_{k-1}M_{k}\widetilde{M}_{k}\dots
    = W_{k}M_{k-1}M_{k}M_{k-1}{M_{k}}^{--}\dots.
\end{align*}
Since $M_{k}M_{k-1}$ is primitive, Lemma 7.1 in \cite{BK19} implies that the first $(\FPG{k}+1)$ subwords of length $\FPG{k}$ are mutually distinct. Thus,
$r(\FPG{k},\x)=2\FPG{k}+1.$
Since $x_{1}^{\FPGW{k}-1}=x_{q_{k}+q_{k-1}+1}^{q_{k}+q_{k-1}+\FPGW{k}-1}$,
$r(\FPGW{k}-1,\x)\le q_{k}+q_{k-1}+\FPGW{k}-1.$
The fact that $r(n+1,\x)\ge r(n,\x)+1$ for any $n\ge1$ gives
\[r(n,\x)=n+\FPG{k}+1\] 
for $\FPG{k}\le n \le \FPGW{k}-1.$
Note that
\begin{align*}
    \x
    &=W_{k+1}M_{k+1}\widetilde{M}_{k+1}\dots
    = W_{k}M_{k-1}{M_{k}}^{a_{k+1}}M_{k-1}M_{k}{M_{k+1}}^{--}\dots.
\end{align*}
If $a_{k+1}=1$,
then $x_{1}^{\FPGW[2]{k}+q_{k-1}-1}=x_{q_{k}+q_{k-1}+1}^{q_{k}+q_{k-1}+\FPGW[2]{k}+q_{k-1}-1}$.
It implies
$r(\FPGW[2]{k}+q_{k-1}-1,\x)\le \FPGW[3]{k}+2q_{k-1}-1.$
Since $r(n+1,\x)\ge r(n,\x)+1$ for any $n\ge1$,
\[r(n,\x)=n+\FPG{k}+1\] 
for $\FPG{k}\le n \le \FPW[k]+q_{k+1}-1.$
Hence,
\[ \Lambda(\x) \cap [\FPG{k}, \FPG{k+1}-1]=\{\FPG{k}\}.
\]
Now, let $a_{k+1}>1.$
Since $\x=W_{k}M_{k-1}M_{k}M_{k}M_{k-1}\dots,$ 
$r(\FPGW{k}\x)\ge r(\FPGW{k}-1,\x)+2.$
From Theorem 2.4 and Lemma 5.3 in \cite{BK19},
$r(\FPGW{k},\x)= 2\FPGW{k}+1.$
On the other hand, note that
\begin{align*}
    \x
    &=W_{k+1}M_{k+1}\widetilde{M}_{k+1}\dots
    = W_{k}M_{k-1}{M_{k}}^{a_{k+1}}M_{k-1}\widetilde{M}_{k+1}^{--}\dots \\
    &= W_{k}M_{k-1}M_{k}{M_{k}}^{a_{k+1}-1}M_{k-1}\widetilde{M}_{k+1}^{--}\dots
    = W_{k}M_{k-1}M_{k}{M_{k}}^{a_{k+1}}M_{k-1}^{--}D'_{k}M_{k+1}^{--}\dots.
\end{align*}
Since $x_{|W_{k}|+q_{k-1}+1}^{|W_{k}|+q_{k-1}+\FPG[a_{k+1}]{k}-1}=x_{|W_{k}|+q_{k}+q_{k-1}+1}^{|W_{k}|+q_{k}+q_{k-1}+\FPG[a_{k+1}]{k}-1}$,
$r(\FP[k]-1,\x)\le|W_{k}|+q_{k}+q_{k-1}+\FP[k]-1.$
The fact that $r(n+1,\x)\ge r(n,\x)+1$ for any $n\ge1$ gives 
\[r(n,\x)=n+\FPGW{k}+1\] 
for $\FPGW{k}\le n \le \FP[k]-1.$
Moreover,
$r(\FP[k],\x)\ge r(\FP[k]-1,\x)+2.$
From Theorem 2.4 and Lemma 5.3 in \cite{BK19}, 
$r(\FP[k],\x)= 2\FP[k]+1.$
Since $\x\in\C_{k+1}^{\I}$, the argument used at level $k$ in (1) implies
$r(\FPG{k+1},\x)=\FPG[2]{k+1}.$
The fact that $r(n+1,\x)\ge r(n,\x)+1$ for any $n\ge1$ gives
\[r(n,\x)=n+\FP[k]+1\] 
for $\FP[k]\le n \le \FPG{k+1}.$
Hence,
\[\Lambda(\x) \cap [\FPG{k}, \FPG{k+1}-1]=\{\FPG{k}, \FPGW{k}, \FP[k]\}.
\]
(6) Suppose that $\x\in\C_{k}^{\II}\cap\C_{k+1}^{\II}$.
Use the same argument with (5). 
Since $\x\in\C_{k}^{\II}$, we have
\[r(n,\x)=n+\FPG{k}+1\] 
for $\FPG{k}\le n \le \FPGW{k}-1.$
Note that
\begin{align*}
    \x
    &=W_{k+1}M_{k}M_{k+1}\widetilde{M}_{k+1}\dots
    = W_{k}M_{k-1}M_{k}{M_{k}}^{a_{k+1}}M_{k-1}M_{k}{M_{k+1}}^{--}\dots \\
    &= W_{k}M_{k-1}M_{k}{M_{k}}^{a_{k+1}+1}{M_{k-1}}^{--}D'_{k}{M_{k+1}}^{--}\dots.
\end{align*}
Since $\x=W_{k}M_{k-1}M_{k}M_{k}M_{k-1}\dots,$
$r(\FPGW{k},\x)\ge r(\FPGW{k}-1,\x)+2.$
From Theorem 2.4 and Lemma 5.3 in \cite{BK19}, 
$r(\FPGW{k},\x)= 2\FPGW{k}+1.$
Moreover, since $x_{|W_{k}|+q_{k-1}+1}^{|W_{k}|+q_{k-1}+\FPG[a_{k+1}+1]{k}-1}=x_{\FPGW{k}+2}^{\FPGW{k}+1+\FPG[a_{k+1}+1]{k}-1}$,
$r(\FPG{k+1}-1,\x)\le|W_{k}|+q_{k}+q_{k-1}+\FPG{k+1}-1.$
The fact that $r(n+1,\x)\ge r(n,\x)+1$ for any $n\ge1$ gives
\[r(n,\x)=n+\FPGW{k}+1\] 
for $\FPGW{k}\le n \le \FPG{k+1}-1.$
Hence,
\[\Lambda(\x) \cap [\FPG{k}, \FPG{k+1}-1]=\{\FPG{k}, \FPGW{k}\}.
\]
\end{proof}

\begin{rmk}
    When $\x\in\C_{1}^{\III}$, Lemma \ref{PositionToComputeGeneralcase} does not determine the elements in $\Lambda(\x)\cap[\FPG{1},\FPG{2}-1]$.
    Thus, we should check how the elements of $\Lambda(\x)\cap[\FPG{1},\FPG{2}-1]$ are expressed in terms of $q_{k}$'s and $|W_{k}|$'s.
    If $\x\in\C_{1}^{\III}$, then $\x$ starts with $0^{a_{1}}10^{a_{1}-1}$.
    Thus, $r(n,\x)=n+1$ for $1\le n\le \FPW[0]-1$.
    Moreover, since $\x\in\C_{2}^{\I}\cup\C_{2}^{\II}$, $\x$ starts with $0^{a_{1}}10^{a_{1}}$ or $0^{a_{1}}10^{a_{1}-1}1$.
    Thus, $r(\FPW[0],\x)=2\FPW[0]+1$.
    Hence, we can follow the proof of Lemma \ref{PositionToComputeGeneralcase} (3) and (4).
    We have $\Lambda(\x)\cap[\FPG{1},\FPG{2}-1]=\{\FPW[0],\FP[1]\}$ for $\x\in\C_{2}^{\I}$ and $\Lambda(\x)\cap[\FPG{1},\FPG{2}-1]=\{\FPW[0]\}$ for $\x\in\C_{2}^{\II}$.
\end{rmk}

For $l=1,2,\dots,6$, define 
\[\Lambda_{l}(\x):=\{n\in\Lambda(\x): \textrm{$n$ appears in ($l$) of Lemma \ref{PositionToComputeGeneralcase}}\}.
\]
It is obvious that $\Lambda_{l}(\x)$'s are mutually distinct and $\cup_{l=1}^{6}{\Lambda_{l}(\x)}=\Lambda(\x)\cap[\FPG{1},\infty).$
Now, we find the minimum of $\LLL(\theta)$ where $\theta$ has bounded partial quotients.

\begin{thm}\label{Proof_TheMinOfSpec}
Let $\theta=[0;a_{1},a_{2},\dots]$ have bounded partial quotients.
We have
$$
\min\LLL(\theta) = \lf_{k\to\infty} \, [1; 1+ a_k, a_{k-1}, a_{k-2}, \dots, a_1].
$$
\end{thm}

\begin{proof}
Let $\x$ be a Sturmian word of slope $\theta$. 
For $k\ge1$, set $\eta_{k}:=\dfrac{q_{k-1}}{q_{k}}, t_k:=\frac{|W_{k}|}{q_{k}}, \epsilon_{k}:=\frac{1}{q_{k}}.$
Note that $t_{k}\le1$, $\lim\limits_{k\to\infty}{\epsilon_{k}}=0$, and $\eta_{k}\ge\epsilon_{k}$ for $k\ge1$.
Set $m_{\theta}:=\varliminf\limits_{i\to\infty}{a_{i}}, M_{\theta}:=\varlimsup\limits_{i\to\infty}{a_{i}}.$
Let $\liminf\limits_{k\to\infty}{\eta_{k}}=[0;b_{1},b_{2},\dots]$.

First, assume that $a_i>1$ for infinitely many $i$.
Since $\eta_{k}=[0;a_{k},a_{k-1},\dots,a_{1}]$ and $M_{\theta}\ge2$, we have $b_{1}=M_{\theta}$, $\liminf\limits_{k\to\infty}{\eta_{k}}<\frac{1}{2}$.
Using Lemma \ref{PositionToComputeGeneralcase}, let us prove that $\liminf\limits_{n_{i}\in\Lambda(\x)}{\frac{n_i}{n_{i+1}}}\ge \liminf\limits_{k\to\infty}{[0;1+a_{k+1},a_{k},\dots]}$ through the 7 cases below.
Note that $\liminf\limits_{k\to\infty}{\eta_{k}}>\dfrac{\liminf\limits_{k\to\infty}{\eta_{k}}}{1+\liminf\limits_{k\to\infty}{\eta_{k}}}=\liminf\limits_{k\to\infty}{\frac{\eta_{k}}{1+\eta_{k}}}=\liminf\limits_{k\to\infty}{[0;1+a_{k},a_{k-1},\dots]}$.

(1) For any $n_i\in\Lambda_{1}(\x)$ with $n_{i+1}\in\Lambda_{1}(\x)\cup\Lambda_{2}(\x)$, $\frac{n_i}{n_{i+1}}$ is
$$\frac{t_{k}+t+\eta_{k}-\epsilon_{k}}{t+1+\eta_{k}-\epsilon_{k}}
\textrm{, }\frac{t+1+\eta_{k}-\epsilon_{k}}{t+1+t_{k}+\eta_{k}-\epsilon_{k}}
\textrm{, }\frac{\eta_{k+1}(t+1+t_{k}+\eta_{k}-\epsilon_{k})}{1-\epsilon_{k+1}}
\textrm{, or }\frac{1-\epsilon_{k+1}}{t'+t_{k+1}+\eta_{k+1}-\epsilon_{k+1}}
$$
for some $k=k(i)$ and $t,t'$ satisfying $W_{k+1}=W_{k}M_{k}^{t}M_{k-1}, W_{k+2}=W_{k+1}M_{k+1}^{t'}M_{k}$.
We have
\[\liminf\limits_{k\to\infty}{\frac{t_{k}+t+\eta_{k}-\epsilon_{k}}{t+1+\eta_{k}-\epsilon_{k}}},\; \liminf\limits_{k\to\infty}{\frac{t+1+\eta_{k}-\epsilon_{k}}{t+1+t_{k}+\eta_{k}-\epsilon_{k}}}\ge\frac{t}{t+1}\ge \frac{1}{2}>\liminf\limits_{k\to\infty}{\eta_{k}},\]
\[\liminf\limits_{k\to\infty}{\frac{\eta_{k+1}}{1-\epsilon_{k+1}}(t+1+t_{k}+\eta_{k}-\epsilon_{k})}\ge \liminf\limits_{k\to\infty}{2\eta_{k+1}}\ge\liminf\limits_{k\to\infty}{\eta_{k}},\]
and
\[ \liminf\limits_{k\to\infty}{\frac{1-\epsilon_{k+1}}{t_{k+1}+t'+\eta_{k+1}-\epsilon_{k+1}}}=\liminf\limits_{k\to\infty}{\frac{1}{t_{k+1}+t'+\eta_{k+1}}} \ge \liminf\limits_{k\to\infty}{\frac{1}{a_{k+2}+\eta_{k+1}}}=\liminf\limits_{k\to\infty}{\eta_{k+2}}.
\]

(2) For any $n_i\in\Lambda_{2}(\x)$, $\frac{n_i}{n_{i+1}}$ is
\[\frac{t_{k}+t+\eta_{k}-\epsilon_{k}}{t+1+\eta_{k}-\epsilon_{k}}
\textrm{, }\frac{t+1+\eta_{k}-\epsilon_{k}}{t+1+t_{k}+\eta_{k}-\epsilon_{k}}
\textrm{, or }\frac{\eta_{k+1}(t+1+t_{k}+\eta_{k}-\epsilon_{k})}{1+\eta_{k+1}-\epsilon_{k+1}}\]
for some $k=k(i)$ and $t$ satisfying $W_{k+1}=W_{k}M_{k}^{t}M_{k-1}$.
From the previous case,  
\[\liminf\limits_{k\to\infty}{\frac{t_{k}+t+\eta_{k}-\epsilon_{k}}{t+1+\eta_{k}-\epsilon_{k}}},\; \liminf\limits_{k\to\infty}{\frac{t+1+\eta_{k}-\epsilon_{k}}{t+1+t_{k}+\eta_{k}-\epsilon_{k}}}\ge\frac{1}{2}>\liminf\limits_{k\to\infty}{\eta_{k}}.\]
We also have
\[\liminf\limits_{k\to\infty}{\frac{\eta_{k+1}}{1+\eta_{k+1}-\epsilon_{k+1}}(t+1+t_{k}+\eta_{k}-\epsilon_{k})}\ge\liminf\limits_{k\to\infty}{\frac{2\eta_{k+1}}{1+\eta_{k+1}}}\ge\liminf\limits_{k\to\infty}{\eta_{k}}.\]

(3) For any $n_i\in\Lambda_{5}(\x)$ with $n_{i+1}\in\Lambda_{1}(\x)\cup\Lambda_{2}(\x)\cup\Lambda_{5}(\x)$, $\frac{n_i}{n_{i+1}}$ is
\[\frac{1+\eta_{k}-\epsilon_{k}}{1+t_{k}+\eta_{k}-\epsilon_{k}}
\textrm{, }\frac{\eta_{k+1}}{1-\epsilon_{k+1}}(1+t_{k}+\eta_{k}-\epsilon_{k})
\textrm{, or }\frac{1-\epsilon_{k+1}}{t+t_{k+1}+\eta_{k+1}-\epsilon_{k+1}}
\]
for some $k=k(i)$ and $t$ satisfying $W_{k+2}=W_{k+1}M_{k+1}^{t}M_{k}$.
We have 
\[\liminf\limits_{k\to\infty}{\frac{1+\eta_{k}-\epsilon_{k}}{1+t_{k}+\eta_{k}-\epsilon_{k}}}\ge\liminf\limits_{k\to\infty}{\frac{1}{1+t_{k}}}\ge\frac{1}{2}>\liminf\limits_{k\to\infty}{\eta_{k}},\]
\[\liminf\limits_{k\to\infty}{\frac{\eta_{k+1}}{1-\epsilon_{k+1}}(1+t_{k}+\eta_{k}-\epsilon_{k})}\ge\liminf\limits_{k\to\infty}{\eta_{k+1}},\]
and
\[\liminf\limits_{k\to\infty}{\frac{1-\epsilon_{k+1}}{t_{k+1}+t+\eta_{k+1}-\epsilon_{k+1}}}=\liminf\limits_{k\to\infty}{\frac{1}{t_{k+1}+t+\eta_{k+1}}} \ge \liminf\limits_{k\to\infty}{\frac{1}{a_{k+2}+\eta_{k+1}}}=\liminf\limits_{k\to\infty}{\eta_{k+2}}.
\]

(4) For any $n_i\in\Lambda_{6}(\x)$, $\frac{n_i}{n_{i+1}}$ is
\[\frac{1+\eta_{k}-\epsilon_{k}}{1+t_{k}+\eta_{k}-\epsilon_{k}}
\textrm{ or }\frac{\eta_{k+1}}{1+\eta_{k+1}-\epsilon_{k+1}}(1+t_{k}+\eta_{k}-\epsilon_{k})
\]
for some $k=k(i)$.
From the previous case, 
\[\liminf\limits_{k\to\infty}{\frac{1+\eta_{k}-\epsilon_{k}}{1+t_{k}+\eta_{k}-\epsilon_{k}}}> \liminf\limits_{k\to\infty}{\eta_{k}}.\]
We also have
\begin{align*}
    &\liminf\limits_{k\to\infty}{\frac{\eta_{k+1}}{1+\eta_{k+1}-\epsilon_{k+1}}(1+t_{k}+\eta_{k}-\epsilon_{k})}\ge
    \liminf\limits_{k\to\infty}{\frac{\eta_{k+1}}{1+\eta_{k+1}}\liminf\limits_{k\to\infty}{(1+t_{k}+\eta_{k})}} \\
    \ge\;&\frac{\liminf\limits_{k\to\infty}{\eta_{k+1}}}{1+\liminf\limits_{k\to\infty}{\eta_{k+1}}}(1+\liminf\limits_{k\to\infty}{\eta_{k}})=
    \liminf\limits_{k\to\infty}{\eta_{k+1}}.
\end{align*}

(5) For any $n_i\in\Lambda_{3}(\x)$ with $n_{i+1}\in\Lambda_{1}(\x)\cup\Lambda_{2}(\x)\cup\Lambda_{3}(\x)$, $\frac{n_i}{n_{i+1}}$ is \[(1+t_{k}\eta_{k+1}-\epsilon_{k+1})\frac{\eta_{k+2}}{1-\epsilon_{k+2}}
\textrm{ or }\frac{1-\epsilon_{k+2}}{t+t_{k+2}+\eta_{k+2}-\epsilon_{k+2}}
\]
for some $k=k(i)$ and $t$ satisfying $W_{k+3}=W_{k+2}M_{k+2}^{t}M_{k+1}$.
We have
\[\liminf\limits_{k\to\infty}{\frac{(1+t_{k}\eta_{k+1}-\epsilon_{k+1})\eta_{k+2}}{1-\epsilon_{k+2}}}\ge\liminf\limits_{k\to\infty}{\frac{(1-\epsilon_{k+1})\eta_{k+2}}{1-\epsilon_{k+2}}}=\liminf\limits_{k\to\infty}{\eta_{k+2}}\]
and
\[ \liminf\limits_{k\to\infty}{\frac{1-\epsilon_{k+2}}{t_{k+2}+t+\eta_{k+2}-\epsilon_{k+2}}}=\liminf\limits_{k\to\infty}{\frac{1}{t_{k+2}+t+\eta_{k+2}}} \ge \liminf\limits_{k\to\infty}{\frac{1}{a_{k+3}+\eta_{k+2}}}=\liminf\limits_{k\to\infty}{\eta_{k+3}}.
\]

(6) For any $n_i\in\Lambda_{1}(\x)\cup\Lambda_{3}(\x)\cup\Lambda_{5}(\x)$ with $n_{i+1}\in\Lambda_{3}(\x)\cup\Lambda_{4}(\x)$, $\frac{n_i}{n_{i+1}}$ is 
\[\frac{\eta_{k+1}-\epsilon_{k+1}}{1+t_{k}\eta_{k+1}-\epsilon_{k+1}}\]
for some $k=k(i)$.
We have
\begin{align*}
    \liminf\limits_{k\to\infty}{\frac{\eta_{k+1}-\epsilon_{k+1}}{1+t_{k}\eta_{k+1}-\epsilon_{k+1}}} 
    \ge \liminf\limits_{k\to\infty}{\frac{\eta_{k+1}}{1+\eta_{k+1}}}
    =\liminf\limits_{k\to\infty}{[0;1+a_{k+1},a_{k},a_{k-1},\dots]}.
\end{align*}


(7) For any $n_i\in\Lambda_{4}(\x)$, $\frac{n_i}{n_{i+1}}$ is 
\[\frac{\eta_{k+2}}{1+\eta_{k+2}-\epsilon_{k+2}}(1+t_{k}\eta_{k+1}-\epsilon_{k+1})\]
for some $k=k(i)$.
We have
\begin{align*}
    \liminf\limits_{k\to\infty} \frac{\eta_{k+2}(1+t_{k}\eta_{k+1}-\epsilon_{k+1})}{1+\eta_{k+2}-\epsilon_{k+2}}   
    \ge \liminf\limits_{k\to\infty}{\frac{\eta_{k+2}}{1+\eta_{k+2}}}
    =\liminf\limits_{k\to\infty}{[0;1+a_{k+2},a_{k+1},a_{k},\dots]}.
\end{align*} 
Hence, from (1)-(7),
\[\rep(\x)\ge\liminf\limits_{k\to\infty}{[1;1+a_{k},a_{k-1},\dots]}\] 
where $a_i >1$ for infinitely many $i$.

Now, assume that there exists an integer $I>0$ such that $a_{i}=1$ for $i\ge I$.
The assumptions of Lemma \ref{LocatingChainRule} (1) and (2) are not satisfied for any level $k\ge I$. 
In other words, $\Lambda_{1}(\x)\cup\Lambda_{2}(\x)$ is finite.
Thus, it is sufficient to check (3)-(7).
Note that $\liminf\limits_{k\to\infty}{\eta_{k}}=[0;\overline{1}]=\varphi$.
Using $1\ge t_{k}$, $\lim\limits_{k\to\infty}{\epsilon_{k}}=0$,
\[\liminf\limits_{k\to\infty}{\frac{1+\eta_{k}-\epsilon_{k}}{1+t_{k}+\eta_{k}-\epsilon_{k}}}\ge\liminf\limits_{k\to\infty}{\frac{1+\eta_{k}}{2+\eta_{k}}}=\frac{1+\liminf\limits_{k\to\infty}{\eta_{k}}}{2+\liminf\limits_{k\to\infty}{\eta_{k}}}=\varphi=\liminf\limits_{k\to\infty}{\eta_{k}},
\]
(3)-(7) are similarly proved.
Hence, $\rep(\x)\ge[1;2,\overline{1}]$ for a Sturmian word $\x$ of slope $\varphi$.
Therefore, 
\[\rep(\x)\ge\liminf\limits_{k\to\infty}{[1;1+a_{k},a_{k-1},\dots]}\] 
for a Sturmian word $\x$ of slope $\theta$.

The equality holds in the following setting. Choose the sequence $\{k_{j}\}$ such that $\eta_{k_{j}}\to\liminf\limits_{k\to\infty}{\eta_{k}}$, $\lim\limits_{j\to\infty}{(k_{j+1}-k_{j})}=\infty$, and $k_{j+1}-k_{j}$ is odd for all $j$.
Let $\x\in\C_{k_{j}}^{\II}$ for all $k_{j}$ and $\x\in\C_{k_{j}+2l-1}^{\I}\cap\C_{k_{j}+2l}^{\III}$ for all $0<l\le\frac{k_{j+1}-k_{j}-1}{2}$.
Thus, $W_{k_{j}+1}=W_{k_{j}}M_{k_{j}-1}$ for all $k_j$ and $W_{k+1}=W_{k}$ for all $k\neq k_{j}$.
In the proof of (7), we have $\lim\limits_{j\to\infty}{t_{k_{j}}}=0$ and $\textrm{rep}(\x)=\varliminf\limits_{k\to\infty}{[1;1+a_{k},a_{k-1},a_{k-2},\dots]}$. 
In conclusion, $\min\LLL(\theta)=\varliminf\limits_{k\to\infty}{[1;1+a_{k},a_{k-1},a_{k-2},\dots]}$.
\end{proof}

\section{The spectrum of the exponents of repetition of Fibonacci words}\label{GoldenRatio}

In this section, we investigate $\LLL(\varphi)$ where $\varphi:=\frac{\sqrt{5}-1}{2}=[0;\bar{1}]$.
In what follows, assume that $\x$ is a Sturmian word of slope $\varphi$, called a \emph{Fibonacci word}.
Note that $M_{k+1}={M_{k}}^{a_{k+1}}M_{k-1}=M_{k}M_{k-1}$ for all $k\ge1$.
The following lemma is a special case of Lemma \ref{LocatingChainRule}.

\begin{lem}\label{LocatingChainRuleForGoldenRatio}
Let $k\ge1$. The following statements hold.

(1) If $\x\in \C_{k}^{\I}$, then $\x\in \C_{k+1}^{\III}$ and $W_{k+1}=W_k$.

(2) If $\x\in \C_{k}^{\II}$, then $\x\in \C_{k+1}^{\I}\cup\C_{k+1}^{\II}$ and $W_{k+1}=W_{k}M_{k-1}$.

(3) If $\x\in \C_{k}^{\III}$, then $\x\in \C_{k+1}^{\I}\cup\C_{k+1}^{\II}$ and $W_{k+1}=W_k$.
\end{lem}

\begin{proof}
For $k\ge1$, let $W_{k}$ be the unique non-empty prefix of $\x$ defined in which case $\x$ belongs to at level $k$.

(1) Since $a_{k}=1$ for all $k\ge1$, the assumption of (1) in Lemma \ref{LocatingChainRule} cannot be satisfied.
Hence, $\x\in\C_{k+1}^{\III}$.
Since $\x$ starts with $W_{k+1}M_{k+1}\widetilde{M}_{k+1}=W_{k+1}M_{k}\widetilde{M}_{k}D'_{k}{M}_{k+1}^{--}$ for the suffix $W_{k+1}$ of $M_{k}$, $W_{k+1}=W_{k}$ by the uniqueness of $W_{k}$.

(2) and (3) are equivalent to (3) and (4) in Lemma \ref{LocatingChainRule} respectively.
\end{proof}


\begin{table}[t]
    \begin{tabular}{ c | c | c }
      \hline
      $k$   &  $k+1$  & The relation between $W_{k+1}$ and $W_{k}$\\
      \hline \hline
      case (i) & case (iii)  &   $W_{k+1} = W_k$ \\
      \hline
      case (ii) & case (i) case (ii) &  $W_{k+1} = W_k M_{k-1}$ \\
      \hline
      case (iii) & case (i) case (ii) & $W_{k+1} = W_k$ \\
      \hline
    \end{tabular}
    \captionof{table}{The relation between $W_{k+1}$ and $W_{k}$}
\end{table}

By Lemma \ref{LocatingChainRuleForGoldenRatio}, only (iii) should follow (i) in the locating chain of $\x$: $\x\in\C_{k}^{\I}$ implies $\x\in\C_{k+1}^{\III}$.
Hence, if $\x\notin\C_{1}^{\III}$, then the locating chain of $\x$ can be expressed as an infinite sequence of (i)(iii) and (ii). 
If $\x\in\C_{1}^{\III}$, then the locating chain of $\x$ is an infinite sequence of (i)(iii) and (ii), except for the first letter (iii).
Let us denote (i)(iii) and (ii) by $a$ and $b$, respectively.

Since only (iii) should follow (i) in the locating chain of $\x$, the assumptions of Lemma \ref{PositionToComputeGeneralcase} (1) and (2) cannot be satisfied. 
We have the following lemma corresponding to Lemma \ref{PositionToComputeGeneralcase}.
Using $q_{k+1}=q_{k}+q_{k-1}$ for all $k\ge1$, Lemma \ref{PositionToComputeGeneralcase} (3)-(6) are equivalent to (1)-(4) of the following lemma, respectively.

\begin{lem}\label{PositionToCompute}
Let $k\ge1$.

(1) If $\x\in \C_{k}^{\I}\cap\C_{k+2}^{\I}$, then  
$\Lambda(\x) \cap [q_{k+1} -1, q_{k+3}-2] = \{ q_{k+1} + |W_k| -1, q_{k+2} -1\}.$

(2) If $\x\in \C_{k}^{\I}\cap\C_{k+2}^{\II}$, then
$\Lambda(\x) \cap [q_{k+1} -1, q_{k+3}-2] = \{ q_{k+1} + |W_k| -1 \}.$

(3) If $\x\in \C_{k}^{\II}\cap\C_{k+1}^{\I}$, then
$\Lambda(\x) \cap [q_{k+1} -1, q_{k+2}-2] = \{ q_{k+1} -1\}.$ 

(4) If $\x\in \C_{k}^{\II}\cap\C_{k+1}^{\II}$, then
$\Lambda(\x) \cap [q_{k+1} -1, q_{k+2}-2] = \{ q_{k+1} -1, q_{k+1} + |W_k| -1 \}.$
\end{lem}

For $l=1,2,3,4$, define \[\Lambda'_{l}(\x):=\{n\in\Lambda(\x): \text{$n$ appears in ($l$) of Lemma \ref{PositionToCompute}}\}.\]
It is obvious that $\Lambda'_{l}(\x)$'s are mutually distinct and $\cup_{l=1}^{4}{\Lambda'_{l}(\x)}=\Lambda(\x)\cap[q_{2}-1,\infty)=\Lambda(\x)$.
From the definition of $\Lambda_{l}(\x)$, $\Lambda'_{l}(\x)=\Lambda_{l+2}(\x)$ for $l=1,2,3,4$ where the slope of $\x$ is $\varphi$.

Note that $\rep(\x)$ is the limit infimum of $(1+\frac{n_{i}}{n_{i+1}})$'s for $n_{i}\in\cup_{l=1}^{4}{\Lambda'_{l}(\x)}$. 
The following lemma says that it is enough to consider the elements of $\Lambda'_{2}(\x)$ and $\Lambda'_{3}(\x)$ to obtain $\rep(\x)$.


\begin{lem}\label{ComputationOnRep}
Suppose that both $a$ and $b$ appear infinitely many in the locating chain of $\x$. 
Then, \[\rep(\x)=\liminf\limits_{n_{i}\in\Lambda'_{2}(\x)\cup\Lambda'_{3}(\x)}{\left(1+\frac{n_i}{n_{i+1}}\right)}.\]
\end{lem}

\begin{proof}
First, for each $k$ satisfying $\x\in\C_{k}^{\II}\cap\C_{k+1}^{\I}$, there exists $d(k)>0$ such that $\x\in\C_{k}^{\II}\cap\C_{k+2d(k)+1}^{\II}$ and $\x\in\C_{k+2d-1}^{\I}\cap\C_{k+2d}^{\III}$ for $1\le d\le d(k)$.
By Lemma \ref{PositionToCompute}, $\mathcal{J}_{k}:=\Lambda(\x)\cap[q_{k+1}-1, q_{k+2d(k)+2}-2]\subset\Lambda'_{1}(\x)\cup\Lambda'_{2}(\x)\cup\Lambda'_{3}(\x)$.
Note that
$W_{k+1}=\cdots=W_{k+2d(k)+1}$ and 
\[\frac{q_{j+1} +|W_{j}|-1}{q_{j+2}-1}\ge \frac{q_{j+1}-1}{q_{j+2} +|W_{j+1}|-1}\]
for $j=k+1,k+3,\dots,k+2d(k)-1.$ 
Since
\[\frac{q_{j+1}-1}{q_{j+2} +|W_{j+1}|-1}
\] 
is increasing for $j=k,k+1,\dots,k+2d(k)-2$, \[\min_{n_{i}\in\mathcal{J}_{k}}{\frac{n_i}{n_{i+1}}}=\min\bigg\{\frac{q_{k+1}-1}{q_{k+2} +|W_{k+1}|-1}, \frac{q_{k+2d(k)} +|W_{k+2d(k)-1}|-1}{q_{k+2d(k)+2}-1}\bigg\}=\min_{n_{i}\in(\Lambda'_{2}(\x)\cup\Lambda'_{3}(\x))\cap\mathcal{J}_{k}}{\frac{n_i}{n_{i+1}}}.
\]

Second, for each $l$ satisfying $\x\in\C_{l}^{\III}\cap\C_{l+1}^{\II}$, there exists $d'(l)>0$ such that $\x\in\C_{l}^{\III}\cap\C_{l+d'(l)+1}^{\I}$ and $\x\in\C_{j}^{\II}$ for $l+1\le j\le l+d'(l)$.
By Lemma \ref{PositionToCompute}, $\mathcal{J}_{l}^{'}:=\Lambda(\x)\cap[q_{l}-1, q_{l+d'(l)+2}-2]\subset\Lambda'_{2}(\x)\cup\Lambda'_{3}(\x)\cup\Lambda'_{4}(\x)$.
Note that 
\[\frac{q_{l} +|W_{l-1}|-1}{q_{l+2}-1}\le\frac{q_{l+1}-1}{q_{l+2}-1}\textrm{ and }
\frac{q_{j+1}-1}{q_{j+2}-1}\le\frac{q_{j+1}-1}{q_{j+1} +|W_{j}|-1}, \frac{q_{j+1} +|W_{j}|-1}{q_{j+2}-1}
\]
for $j=l+1,\dots,l+d'(l)-1$.
Since 
\[
\frac{q_{j}-1}{q_{j+1}-1}
\]
is increasing for $j=l+1,\dots,l+d'(l)$,
\[\min_{n_{i}\in\mathcal{J}_{l}^{'}}{\frac{n_i}{n_{i+1}}}=\min\bigg\{\dfrac{q_{l} +|W_{l-1}|-1}{q_{l+2}-1}, \dfrac{q_{l+d'(l)+1}-1}{q_{l+d'(l)+2}+|W_{l+d'(l)+1}|-1}\bigg\}=\min_{n_{i}\in(\Lambda'_{2}(\x)\cup\Lambda'_{3}(\x))\cap\mathcal{J}_{l}^{'}}{\frac{n_i}{n_{i+1}}}.
\]
Since $\Lambda(\x)$ is the union of $\mathcal{J}_{k}$'s and $\mathcal{J}_{l}^{'}$'s, $\rep(\x)=\liminf\limits_{n_{i}\in\Lambda'_{2}(\x)\cup\Lambda'_{3}(\x)}{\left(1+\frac{n_i}{n_{i+1}}\right)}$.
\end{proof}

Let $d$ be a positive integer. 
We define \emph{$a$-chain} to be a subword $aa\cdots a$ in the locating chain of $\x$ which $b$ appears before and after.
For example, if the locating chain of $\x=abbaaabaab\dots$, then $a$-chains are $a, aaa, aa, \dots$. 
Similarly, \emph{$b$-chain} is defined by a subword $bb\cdots b$ in the locating chain of $\x$ which $a$ appears before and after.
We say that an $a$-chain or a $b$-chain is a \emph{chain}.
From the definition of a chain, $a$-chains and $b$-chains alternatively appear in the locating chain of $\x$.
We can choose two sequences $\{m_i{(\x)}\}_{i\ge1}$ and $\{l_j{(\x)}\}_{j\ge1}$ defined as follows: 
Let $m_{i}{(\x)}$ (resp., $l_j{(\x)}$) be the length of the $i$th $a$-chain (resp., the $j$th $b$-chain) in the locating chain of $\x$. 
Let $b^{l_1{(\x)}}$ follow $a^{m_{1}{(\x)}}$. 
In other words, the locating chain of $\x$ is $c(\x)a^{m_{1}{(\x)}}b^{l_{1}{(\x)}}a^{m_{2}{(\x)}}b^{l_{2}{(\x)}}\dots.$ for the unique finite word $c(\x)$.
For example, if the locating chain of $\x$ is $\text{(iii)}bbabbabaabbba\dots$, then $c(\x)=\text{(iii)}bb$, $m_{1}{(\x)}=1$, $m_{2}{(\x)}=1$, $m_{3}{(\x)}=2$, $l_{1}{(\x)}=2$, $l_{2}{(\x)}=1$, $l_{3}{(\x)}=3$.
Let
\[
S_{d}:=\{\x :\text{$m_i{(\x)}\ge d$ for infinitely many $i$ or $l_j{(\x)}\ge d$ for infinitely many $j$}\}.
\]
By definition, $S_{d+1}\subset S_{d}$. 
In what follows, we will write $m_i{(\x)}\text{ and }l_j{(\x)}$ simply $m_i\text{ and }l_j$, when no confusion can arise.
Now, let us prove Theorem \ref{TheMaxMinForGoldenRatio}, \ref{TheRightGap}, \ref{TheThirdRightGap} and Proposition \ref{Cardinality_Of_Sturmian words}.

\begin{thm}\label{Proof_TheMaxMinForGoldenRatio}
Let $\x$ be a Sturmian word of slope $\varphi$. 
Then, $\mu_{min} \le \rep(\x) \le\mu_{max}.$
Moreover, the locating chain of $\x$ is $u \overline{a}$ or $v\overline{b}$ for some finite words $u,v$ if and only if $\rep(\x) = \mu_{max}.$
We have $\x\in S_d$ for any $d\ge1$ if and only if $\rep(\x) = \mu_{min}.$
\end{thm}

\begin{proof}
First, assume that there exists a constant $K$ such that $\x\in\C_{k}^{\II}$ for all $k\ge K$.
By Lemma \ref{PositionToCompute}, $\Lambda(\x)\cap[q_{K+1}-1,\infty)=\{q_{k}-1,q_{k}+|W_{k-1}|-1:k\ge K+1 \}$.
Since $W_{k+1}=W_{k}M_{k-1}$ for any $k\ge K$,
\begin{align*}
    \liminf_{n_i \in\Lambda(\x)}{\frac{n_i}{n_{i+1}}}
    &=\liminf_{k\ge K+1}{\left\{\frac{q_{k}-1}{q_{k}+|W_{k-1}|-1},\frac{q_{k}+|W_{k-1}|-1}{q_{k+1}-1}\right\}}\\
    &=\min\left\{\liminf_{k\ge K+1}{\left(\frac{q_{k}-1}{q_{k}+|W_{k-1}|-1}\right)}, \liminf_{k\ge K+1}{\left(\frac{q_{k}+|W_{k-1}|-1}{q_{k+1}-1}\right)}\right\}
    =\min\{\varphi, 1\}=\varphi.
\end{align*}
Hence, $\rep(\x)=\liminf\limits_{n_i \in\Lambda(\x)}{\left(1+\frac{n_i}{n_{i+1}}\right)}=\mu_{max}$. 

Second, assume that there exists a constant $K$ such that $\x\in\C_{K+2l}^{\I}\cap\C_{K+2l+1}^{\III}$ for all $l\ge 0$.
By Lemma \ref{PositionToCompute}, $\Lambda(\x)\cap[q_{K+1}-1,\infty)=\{q_{K+2l+1}+|W_{K+2l}|-1, q_{K+2l+2}-1:l\ge 0 \}$.
Since $|W_{k}|$ is constant for $k\ge K$,
\begin{align*}
    \liminf_{n_i \in\Lambda(\x)}{\frac{n_i}{n_{i+1}}}
    &=\liminf_{l\ge 0}{\left\{\frac{q_{k+2l+1}+|W_{k+2l}|-1}{q_{K+2l+2}-1},\frac{q_{K+2l+2}-1}{q_{k+2l+3}+|W_{k+2l+2}|-1}\right\}}\\
    &=\min\left\{\liminf_{l\ge 0}{\left(\frac{q_{k+2l+1}+|W_{k+2l}|-1}{q_{K+2l+2}-1}\right)},\liminf_{l\ge 1}{\left(\frac{q_{K+2l}-1}{q_{k+2l+1}+|W_{k+2l}|-1}\right)}\right\}
    =\varphi.
\end{align*}
Hence, $\rep(\x)=\liminf\limits_{n_i \in\Lambda(\x)}{\left(1+\frac{n_i}{n_{i+1}}\right)}=\mu_{max}$.

Now, let both $a$ and $b$ occur infinitely many in the locating chain of $\x$.
Since $ba$ appears infinitely many in the locating chain of $\x$, we can choose an infinite sequence $\{n_{i(j)}\}_{j\ge1}\subset\Lambda'_{3}(\x)$. For each $j\ge 1$, Lemma \ref{PositionToCompute} gives $n_{i(j)} = q_{k+1}-1$,$n_{i(j)+1}=q_{k+2}+|W_{k+1}|-1$ for some $k=k(j).$
Note that $W_{k(j)+1}=W_{k(j)}M_{k(j)-1}$
By definition, 
\begin{align*}
\rep(\x)
&\le \liminf_{j\to\infty}{\left(1+\dfrac{q_{k(j)+1}-1}{q_{k(j)+2}+|W_{k(j)+1}|-1}\right)}\\
&\le\liminf_{j\to\infty}{\left(1+\dfrac{q_{k(j)+1}-1}{q_{k(j)+2}+q_{k(j)-1}-1}\right)}
=1+\dfrac{\varphi}{1+\varphi^3}<\mu_{max}.
\end{align*}
Hence, $\rep(\x)< \mu_{max}$.
In other words, $\rep(\x)= \mu_{max}$ implies that the locating chain of $\x$ is $u \overline{a}$ or $v\overline{b}$ for some finite words $u,v$.

Finally, let us use Lemma \ref{ComputationOnRep} to find out the minimum of $\LLL(\varphi)$. 
For $n_{i}\in\Lambda'_{2}(\x)$, let $n_i = q_{k}+|W_{k-1}|-1$, $n_{i+1}=q_{k+2}-1$ for some $k=k(i).$
For $n_{i}\in\Lambda'_{3}(\x)$, let $n_i = q_{k'+1}-1$, $n_{i+1}=q_{k'+2}+|W_{k'+1}|-1$ for some $k'=k'(i).$
From Lemma \ref{ComputationOnRep},
\begin{align*}
    \rep(\x)
    &=\min\left\{\liminf\limits_{n_{i}\in\Lambda'_{2}(\x)}{\left(1+\frac{n_{i}}{n_{i+1}}\right)},\liminf\limits_{n_{i}\in\Lambda'_{3}(\x)}{\left(1+\frac{n_{i}}{n_{i+1}}\right)}\right\} \\
    &=\min\left\{\liminf_{i\to\infty}{\left(1+\dfrac{q_{k(i)}+|W_{k(i)-1}|-1}{q_{k(i)+2}-1}\right)},\liminf_{i\to\infty}{\left(1+\dfrac{q_{k'(i)+1}-1}{q_{k'(i)+2}+|W_{k'(i)+1}|-1}\right)}\right\}\\
    &\ge\min\left\{\liminf_{i\to\infty}{\left(1+\dfrac{q_{k(i)}-1}{q_{k(i)+2}-1}\right)},\liminf_{i\to\infty}{\left(1+\dfrac{q_{k'(i)+1}-1}{q_{k'(i)+2}+q_{k'(i)+1}-1}\right)}\right\}\\
    &=\min\left\{1+\varphi^2,1+\dfrac{\varphi}{\varphi+1}\right\}
    =\mu_{min}.
\end{align*}
Hence, $\rep(\x)\ge \mu_{min}$.
Moreover, if $\rep(\x)=\mu_{min}$, then \[\liminf\limits_{i\to\infty}{\dfrac{|W_{k(i)-1}|}{q_{k(i)}}}=0\textrm{ or }\limsup\limits_{i\to\infty}{\dfrac{|W_{k'(i)+1}|}{q_{k'(i)+1}}}=1.\]
Thus, arbitrarily long sequence $aa\dots a$ or $bb\dots b$ should occur in the locating chain of $\x$, i.e. $\x\in S_{d}$ for any $d\ge 1$.
\end{proof}


\begin{thm}\label{Proof_TheRightGap}
The intervals
$
\left( \mu_{2} , \mu_{max} \right),
\left(\mu_{3}, \mu_{2}\right)$
are maximal gaps in $\LLL(\varphi)$. Moreover, the locating chain of $\x$ is $u\overline{ab}$ for some finite word $u$ if and only if $\rep(\x) = \mu_{2}$. 
The locating chain of $\x$ is $v\overline{b^{2}a^{2}}$ for some finite word $v$ if and only if $\rep(\x)=\mu_{3}$.
\end{thm}

\begin{proof}
If $\rep(\x)<\mu_{max}$, then $\x\in S_{1}$ by Theorem \ref{TheMaxMinForGoldenRatio}.

First, let $\x\in S_1 \cap S_{2}^{c}$.
Since any chains of length greater than 1 occur at most finitely many in the locating chain of $\x$, there exists an integer $I>0$ satisfying $m_j=l_j=1$ for $j\ge I$.
Thus, the locating chain of $\x$ is $u\overline{ab}$ for some finite word $u$. 
Using Lemma \ref{PositionToCompute} and ~\ref{ComputationOnRep}, we obtain $\rep(\x)=\mu_{2}$.

Second, let $\x\in S_2$.
Using Lemma \ref{PositionToCompute} and \ref{ComputationOnRep}, $\rep(\x)=\mu_3$ where the locating chain of $\x$ is $v\overline{a^2 b^2}$ for some finite word $v$.
If $m_j\ge 3$ for infinitely many $j$, then there exists an infinite sequence $\{k(j)\}$ satisfying $\x\in\C_{k(j)-6}^{\I}\cap\C_{k(j)-4}^{\I}\cap\C_{k(j)-2}^{\I}\cap\C_{k(j)}^{\II}$ for all $j$.
By Lemma \ref{ComputationOnRep}, 
\[\rep(\x)
\le\liminf\limits_{j\to\infty}{\bigg(1+\frac{q_{k(j)-1}+|W_{k(j)-6}|-1}{q_{k(j)+1}-1}\bigg)}
\le 1+\varphi^2+\varphi^7
\] 
where $W_{k(j)}=W_{k(j)-6}$ for all $j$.
If $l_j\ge 3$ for infinitely many $j$, then there exists an infinite sequence $\{k(j)\}$ satisfying $\x\in\C_{k(j)-2}^{\II}\cap\C_{k(j)-1}^{\II}\cap\C_{k(j)}^{\II}\cap\C_{k(j)+1}^{\I}$ for all $j$.
By Lemma \ref{ComputationOnRep}, 
\begin{align*}
    \rep(\x)
    &\le\liminf\limits_{j\to\infty}{\bigg(1+\frac{q_{k(j)+1}-1}{q_{k(j)+2}+|W_{k(j)+1}|-1}\bigg)}\\
    &=\liminf\limits_{j\to\infty}{\bigg(1+\frac{q_{k(j)+1}-1}{q_{k(j)+2}+q_{k(j)-1}+q_{k(j)-2}+q_{k(j)-3}+|W_{k(j)-2}|-1}\bigg)}\\
    &\le\liminf\limits_{j\to\infty}{\bigg(1+\frac{q_{k(j)+1}-1}{q_{k(j)+2}+2q_{k(j)-1}-1}\bigg)}
    = 1+\frac{\varphi}{1+2\varphi^3}.
\end{align*}
Since $1+\varphi^2+\varphi^7, 1+\dfrac{\varphi}{1+2\varphi^3}< \mu_{3}$, $\rep(\x)<\mu_{3}$ for $\x\in S_{3}$.
Now, let $\x\in S_{2}\cap S_{3}^{c}$. 
By definition, there exists an integer $I>0$ such that $m_j, l_j\le2$ for $j\ge I$.
If $l_j=1$, $m_{j+1}=2$ for infinitely many $j$, then there exists an infinite sequence $\{k(j)\}$ satisfying $\x\in\C_{k(j)-7}^{\I}\cap\C_{k(j)-5}^{\II}\cap\C_{k(j)-4}^{\I}\cap\C_{k(j)-2}^{\I}\cap\C_{k(j)}^{\II}$ for all $j$.
By Lemma \ref{ComputationOnRep},
\[\rep(\x)
\le\liminf\limits_{j\to\infty}{\bigg(1+\frac{q_{k(j)-1}+q_{k(j)-6}+|W_{k(j)-7}|-1}{q_{k(j)+1}-1}\bigg)}
\le 1+\varphi^2+\varphi^7+\varphi^8
\]
where $W_{k(j)}=W_{k(j)-4}$, $W_{k(j)-5}=W_{k(j)-7}$, and $W_{k(j)-4}=W_{k(j)-5}M_{k(j)-6}$ for all $j$.
If $m_j=1$, $l_{j}=2$ for infinitely many $j$, then there exists an infinite sequence $\{k(j)\}$ satisfying $\x\in\C_{k(j)-4}^{\II}\cap\C_{k(j)-3}^{\I}\cap\C_{k(j)-1}^{\II}\cap\C_{k(j)}^{\II}\cap\C_{k(j)+1}^{\I}$ for all $j$.
By Lemma \ref{ComputationOnRep},
\begin{align*}
    \rep(\x)
    &\le\liminf\limits_{j\to\infty}{\bigg(1+\frac{q_{k(j)+1}-1}{q_{k(j)+2}+|W_{k(j)+1}|-1}\bigg)}\\
    &=\liminf\limits_{j\to\infty}{\bigg(1+\frac{q_{k(j)+1}-1}{q_{k(j)+2}+q_{k(j)-1}+q_{k(j)-2}+|W_{k(j)-3}|-1}\bigg)}\\
    &\le\liminf\limits_{j\to\infty}{\bigg(1+\frac{q_{k(j)+1}-1}{q_{k(j)+2}+q_{k(j)}+q_{k(j)-5}-1}\bigg)}
    = 1+\frac{\varphi}{\varphi^7+\varphi^2+1}
\end{align*}
where $W_{k(j)-1}=W_{k(j)-3}$ and $W_{k(j)-3}=W_{k(j)-4}M_{k(j)-5}$ for all $j$.
Since $1+\varphi^2+\varphi^7+\varphi^8, 1+\dfrac{\varphi}{\varphi^7+\varphi^2+1}<\mu_{3}$, $\rep(\x)<\mu_{3}$ where the locating chain of $\x$ is not $v\overline{a^2b^2}$ for some finite word $v$.
Hence, $\max\{\rep(\x):\x\in S_{2}\}=\mu_3$.
Moreover, $\rep(\x)=\mu_{3}$ if and only if the locating chain of $\x$ is $v\overline{a^2b^2}$ for some finite word $v$.
Therefore, two intervals $(\mu_{2},\mu_{max})$, $(\mu_{3},\mu_{2})$ are maximal gaps in $\LLL(\varphi)$.
On the other hand, by Theorem \ref{TheMaxMinForGoldenRatio}, $\rep(\x)=\mu_{max}$ if and only if $\x\in S_{1}^{c}$.
In the proof above, $\rep(\x)=\mu_{2}$ for $\x\in S_{1}\cap S_{2}^{c}$, and $\rep(\x)\le\mu_{3}$ for $\x\in S_{2}$.
Hence, $\rep(\x)=\mu_{2}$ if and only if $\x\in S_{1}\cap S_{2}^{c}$.
\end{proof}



\begin{thm}\label{Proof_TheThirdRightGap}
The interval $\left(\mu_{4}, \mu_{3}\right)$ is a maximal gap in $\LLL(\varphi)$.
Moreover, $\rep(\x)=\mu_{4}$ if and only if $\x \in S_{2}\cap S_{3}^{c}$ satisfies the following two conditions:\\
1) The locating chain of $\x$ is $u(b^2 a^2)^{e_1}ba(b^2 a^2)^{e_2}ba\dots$ for some finite word $u$ and integers $e_i \ge 1$.\\
2) $\limsup\limits_{i\ge1}\{e_{i}\}=\infty$.\\
Furthermore, $\mu_{4}$ is a limit point of $\LLL(\varphi)$.
\end{thm}

\begin{proof}
From Lemma \ref{PositionToCompute} and \ref{ComputationOnRep}, $\rep(\x)=\mu_{4}$ if $x\in S_{2}\cap S_{3}^{c}$ satisfies the above two conditions 1) and 2).
Assume that a Sturmian word $\x$ satisfies $\rep(\x)\in\left(\mu_{4},\mu_{3}\right)$.
In the proof of Theorem \ref{TheRightGap}, $\rep(\x)\ge\mu_{2}$ for $\x \in S_{2}^{c}$, and $\rep(\x)\le\min\{1+\varphi^2+\varphi^7, 1+\dfrac{\varphi}{2\varphi^3+1}\}<\mu_{4}$ for $\x\in S_{3}$.
Thus, $\x\in S_{2}\cap S_{3}^{c}$.
By definition, there exists an integer $I>0$ such that $m_{j}, l_j \le 2$ for all $j\ge I$.
Moreover, $\rep(\x)<\mu_{3}$ implies that $\{j:m_{j}=1\}\cup\{j:l_{j}=1\}$ is infinite. 
Hence, $\{j:m_{j}=1,l_{j}=2\}\cup\{j:l_{j}=1,m_{j+1}=2\}$ is infinite.

First, if $\{j:l_{j}=1,m_{j+1}=2\}$ is infinite, then $\rep(\x)<1.432<\mu_{4}$.
Thus, $\{j:l_{j}=1,m_{j+1}=2\}$ is finite.
In other words, there exists an integer $I'>0$ such that $m_{j}=2$ implies $l_{j-1}=2$ for all $j>I'$.
Since $\{j:m_{j}=1,l_{j}=2\}\cup\{j:l_{j}=1,m_{j+1}=2\}$ is infinite, $\{j:m_{j}=1,l_{j}=2\}$ is infinite.
Now, let us show that both $\{j:m_{j}=1,l_{j-1}=l_{j}=2\}$ and $\{j:m_{j-1}=m_{j}=1,l_{j-1}=1,l_{j}=2\}$ are finite.
If $\{j:m_{j}=1,l_{j-1}=l_{j}=2\}$ is infinite, then there exists an infinite sequence $\{k(j)\}$
such that $\x\in\C_{k(j)-5}^{\II}\cap\C_{k(j)-4}^{\II}\cap\C_{k(j)-3}^{\I}\cap\C_{k(j)-1}^{\II}\cap\C_{k(j)}^{\II}\cap\C_{k(j)+1}^{\I}$ for all $j$.
Hence,
\begin{align*}
    \rep(\x)
    &\le\liminf\limits_{j\to\infty}{\left(1+\frac{q_{k(j)+1}-1}{q_{k(j)+2}+|W_{k(j)+1}|-1}\right)} \\
    &=\liminf\limits_{j\to\infty}{\left(1+\frac{q_{k(j)+1}-1}{q_{k(j)+2}+q_{k(j)-1}+q_{k(j)-2}+q_{k(j)-5}+q_{k(j)-6}+|W_{k(j)-5}|-1}\right)} \\
    &\le\liminf\limits_{j\to\infty}{\left(1+\frac{q_{k(j)+1}-1}{q_{k(j)+2}+q_{k(j)-1}+q_{k(j)-2}+q_{k(j)-5}+q_{k(j)-6}-1}\right)} \\
    &=1+\frac{\varphi}{1+{\varphi^3}+{\varphi^4}+{\varphi^7}+{\varphi^8}}
    <\mu_{4}.
\end{align*}
It follows that $\{j:m_{j}=1,l_{j-1}=l_{j}=2\}$ is finite.

If $\{j:m_{j-1}=m_{j}=1,l_{j-1}=1,l_{j}=2\}$ is infinite, then there exists an infinite sequence $\{k(j)\}$
such that $\x\in\C_{k(j)-7}^{\II}\cap\C_{k(j)-6}^{\I}\cap\C_{k(j)-4}^{\II}\cap\C_{k(j)-3}^{\I}\cap\C_{k(j)-1}^{\II}\cap\C_{k(j)}^{\II}\cap\C_{k(j)+1}^{\I}$ for all $j$.
Hence,
\begin{align*}
    \rep(\x)
    &\le\liminf\limits_{j\to\infty}{\left(1+\frac{q_{k(j)+1}-1}{q_{k(j)+2}+|W_{k(j)+1}|-1}\right)} \\
    &=\liminf\limits_{j\to\infty}{\left(1+\frac{q_{k(j)+1}-1}{q_{k(j)+2}+q_{k(j)-1}+q_{k(j)-2}+q_{k(j)-5}+q_{k(j)-8}+|W_{k(j)-7}|-1}\right)} \\
    &\le\liminf\limits_{j\to\infty}{\left(1+\frac{q_{k(j)+1}-1}{q_{k(j)+2}+q_{k(j)-1}+q_{k(j)-2}+q_{k(j)-5}+q_{k(j)-8}-1}\right)} \\
    &=1+\frac{\varphi}{1+{\varphi^3}+{\varphi^4}+{\varphi^7}+{\varphi^{10}}}
    <\mu_{4}.
\end{align*}
It follows that $\{j:m_{j-1}=m_{j}=1,l_{j-1}=1,l_{j}=2\}$ is finite.
Therefore, both $\{j:m_{j}=1,l_{j-1}=l_{j}=2\}$ and $\{j:m_{j-1}=m_{j}=1,l_{j-1}=1,l_{j}=2\}$ are finite.

Next, let us prove that $\{j:m_j = m_{j+1}=1\}$ is finite. 
Suppose that $\{j:m_j = m_{j+1}=1\}$ is infinite. 
Then, $\{j:m_j = m_{j+1}=1, l_j=1\}\cup\{j:m_j = m_{j+1}=1, l_j=2\}$ is infinite.
Note that $\{j:l_{j}=1,m_{j+1}=2\}$ is finite.
In other words, there exists an integer $I''>0$ such that $l_{j}=1$ implies $m_{j+1}=1$ for all $j>I''$.
Thus, if $\{j:m_j = m_{j+1}=1, l_j=1\}$ is infinite, then $\{j:m_j = m_{j+1}=1, l_j=1, l_{j+1}=2\}\cup\{j:m_j = m_{j+1}=m_{j+2}=1, l_j=l_{j+1}=1\}$ is infinite. 
Since $\{j:m_{j}=m_{j+1}=1,l_{j}=1,l_{j+1}=2\}$ is finite, $\{j:m_j = m_{j+1}=m_{j+2}=1, l_j=l_{j+1}=1\}$ is infinite. 
Thus, $\{j:m_j = m_{j+1}=m_{j+2}=1, l_j=l_{j+1}=1, l_{j+2}=2\}\cup\{j:m_j = m_{j+1}=m_{j+2}=m_{j+3}=1, l_j=l_{j+1}=l_{j+2}=1\}$ is infinite.
Since $\{j:m_{j}=m_{j+1}=1,l_{j}=1,l_{j+1}=2\}$ is finite, $\{j:m_j = m_{j+1}=m_{j+2}=m_{j+3}=1, l_j=l_{j+1}=l_{j+2}=1\}$ is infinite.
By the same argument, it follows that the locating chain of $\x$ is $u\overline{ab}$ for some finite word $u$, which leads a contradiction with $x\in S_2$. 
Hence, $\{j:m_j = m_{j+1}=1, l_j=1\}$ is finite. 
Similarly, we use the same argument to induce that $\{j:m_j = m_{j+1}=1, l_j=2\}$ is finite. 
Therefore, $\{j:m_j = m_{j+1}=1\}$ is finite.

From the above arguments, we have the locating chain of $\x$ is $u(b^2 a^2)^{e_1}b^{f_1} a(b^2 a^2)^{e_2}b^{f_2} a\dots$ for integers $e_i \ge 1$, $f_j = 1\textrm{ or }2$, and some finite word $u$. 
Since $\{j:m_{j}=1,l_{j-1}=l_{j}=2\}$ is finite, we can assume that $f_j=1$ for all $j$. The locating chain of $\x$ is $u(b^2 a^2)^{e_1}ba(b^2 a^2)^{e_2}ba\dots$.
Moreover, if 
$d=\limsup\limits_{i\ge1}\{e_i\}<\infty,$
then 
\[  \rep(\x)
    \le 1+{\left(\frac{1}{\varphi}+\varphi+\frac{\varphi^6}{1-\varphi^{6d+3}}+\frac{\varphi^{10}}{1-\varphi^6}\frac{1-\varphi^{6d}}{1-\varphi^{6d+3}}\right)^{-1}}\\
    <\mu_{4}.
\]
It follows $\limsup\limits_{i\ge1}\{e_i\}=\infty.$
However, for a Sturmian word $\x$ such that the locating chain of $\x$ is $u(b^2 a^2)^{e_1}ba(b^2 a^2)^{e_2}ba\dots$ and $\limsup\limits_{i\ge1}\{e_i\}=\infty$, $\rep(\x)=\mu_{4}$. 
It implies that there does not exist a Sturmian word $\x$ satisfying $\rep(\x)\in\left(\mu_{4},\mu_{3}\right)$. 
Hence, 
$\left(\mu_{4}, \mu_{3}\right)$
is a maximal gap in $\LLL(\varphi)$.
Furthermore,
\begin{align*}
    \rep(\dots\overline{(b^2a^2)^d ba})
    &= 1+{\left(\frac{1}{\varphi}+\varphi+\frac{\varphi^6}{1-\varphi^{6d+3}}+\frac{\varphi^{10}}{1-\varphi^6}\frac{1-\varphi^{6d}}{1-\varphi^{6d+3}}\right)^{-1}}\\
    &\to \mu_{4}\textrm{ as }d\to\infty.
\end{align*}
Hence, $\mu_{4}$ is a limit point of $\LLL(\theta)$.
\end{proof}


\begin{prop}\label{proof_Cardinality_Of_Sturmian_words}
For $\alpha\in\{\mu_{max},\mu_{2},\mu_{3}\}$, there are only countably many Sturmian words $\x$ of slope $\varphi$ satisfying $\rep(\x) = \alpha.$
For $\beta\in\{\mu_{4}, \mu_{min}\}$, there are uncountably many Sturmian words $\x$ of slope $\varphi$ satisfying $\rep(\x) = \beta.$
\end{prop}

\begin{proof}
From Theorem \ref{TheMaxMinForGoldenRatio}, $\rep(\x)=\mu_{max}$ if and only if the locating chain of $\x$ is $u\overline{a}$ or $v\overline{b}$ for some finite words $u,v.$
Thus, $\x$ satisfying $\rep(\x)=\mu_{max}$ is completely determined by the choice of $u$ or $v$.
Hence, there exist only countably many Sturmian words $\x$ of slope $\varphi$ with $\rep(\x)=\mu_{max}.$ 
Theorem \ref{TheMaxMinForGoldenRatio} also implies that 
$\rep(\x)=\mu_{min}$ if and only if $\x\in S_{d}$ for any $d\ge1$. 
Hence, it is possible to choose $2^{\aleph_{0}}$ sequences $\{m_i\}\cup\{l_j\}$. 
Namely, there exist uncountably many Sturmian words $\x$ of slope $\varphi$ with $\rep(\x)=\mu_{min}.$

On the other hand, Theorem \ref{TheRightGap} implies that $\rep(\x)=\mu_{2}$ if and only if the locating chain of $\x$ is $u\overline{ba}$ for some finite word $u.$
Thus, $\x$ satisfying $\rep(\x)=\mu_{2}$ is completely determined by the choice of $u$.
Hence, there exist only countably many Sturmian words $\x$ of slope $\varphi$ with $\rep(\x)=\mu_{2}.$ 
Theorem \ref{TheRightGap} also implies that $\rep(\x)=\mu_{3}$
if and only if the locating chain of $\x$ is $u\overline{b^2a^2}$ for some finite word $u.$
Thus, $\x$ satisfying $\rep(\x)=\mu_{3}$ is completely determined by the choice of $u$.
Hence, there exist only countably many Sturmian words $\x$ of slope $\varphi$ with $\rep(\x)=\mu_{3}.$

Finally, Theorem \ref{TheThirdRightGap} implies that 
$\rep(\x)=\mu_{4}$
if and only if the locating chain of $\x$ is 
$u(b^2 a^2)^{e_1}ba(b^2 a^2)^{e_2}ba\dots$
for some finite word $u$ and integers $e_i \ge 1$ satisfying $\limsup\limits_{i\ge1}\{e_{i}\}=\infty$.
Hence, it is possible to choose $2^{\aleph_{0}}$ sequences $\{e_i\}$. 
Namely, there exist uncountably many Sturmian words $\x$ of slope $\varphi$ with $\rep(\x)=\mu_{4}.$
\end{proof}

\section*{acknowledgement}
    The author wishes to thank Dong Han Kim and Seonhee Lim for introducing me to the subject and very helpful comments.

\end{document}